\documentclass[12pt]{amsart}
\usepackage{epsfig, amsmath, amssymb, latexsym,  amscd, color}
\usepackage{amsthm}
\usepackage{ulem}

\usepackage{times}  
\usepackage{graphicx}
\usepackage{tikz}
\usepackage{verbatim}
\usetikzlibrary{calc}

\usepackage{amsfonts}
\usepackage{amsmath}
\usepackage{epsfig}
\usepackage{epstopdf}

\newcommand{ \pp }{{\mathbb P}}

\newcommand{ \ii }{{\mathcal I}}

\newcommand{ \oo }{{\mathcal O}}

\def\cocoa{{\hbox{\rm C\kern-.13em o\kern-.07em C\kern-.13em o\kern-.15em A}}}

\newtheorem{thm}{Theorem}[section]
 \newtheorem{cor}[thm]{Corollary}
 \newtheorem{lem}[thm]{Lemma}
 \newtheorem{prop}[thm]{Proposition}

  \newtheorem{conj}[thm]{Conjecture}
 \theoremstyle{definition}
 \newtheorem{defn}[thm]{Definition}
 \theoremstyle{remark}
 \newtheorem{rem}[thm]{Remark}
  \newtheorem{ex}[thm]{Example}
  \newtheorem{notn}[thm]{Notation}
  \newtheorem{quest}[thm]{Question}

\begin{document}

\title{Almost maximal growth of the Hilbert function}
\date{}

\author[L. Chiantini]{Luca Chiantini}
\address[L. Chiantini]{Dipartimento di Ingegneria dell'Informazione e Scienze Matematiche 
\\ Universita` di Siena \\ Via Roma 56 (S.Niccolo')\ \\ 53100 Siena,  Italy}
\email{luca.chiantini@unisi.it}
\urladdr{http://www.diism.unisi.it/~chiantini}

\author[J. Migliore]{Juan Migliore${}^{*}$}
\address[J. Migliore]{Department of Mathematics, University of Notre Dame, Notre Dame, IN 46556}
\email{migliore.1@nd.edu}
\urladdr{http://www.nd.edu/~jmiglior/}

\thanks{
${}^*$ Part of the work for this paper was done while this
author was sponsored by the National Security Agency under Grant
Number H98230-12-1-0204, and by the Simons Foundation under grant \#208579. { He also thanks
the Italian CNR-GNSAGA}}


\begin{abstract} 
Let $A = S/J$ be a standard artinian graded algebra over the polynomial ring $S$.  A theorem of Macaulay dictates the possible growth of the Hilbert function of $A$ from any degree to the next, and if this growth is the maximal possible then strong consequences have been given by Gotzmann. It can be phrased in terms of the base locus of the linear system defined by the relevant component(s) of $J$.  If $J$ is the artinian reduction of the ideal of a finite set of points in projective space then this maximal growth for $A$ was shown by Bigatti, Geramita and the second author to imply strong geometric consequences for the points.  We now suppose that the growth of the Hilbert function is one less than maximal.  This again has (not as) strong consequences for the base locus defined by the relevant component.  And when $J$ is the artinian reduction of the ideal of a finite set of points in projective space, we prove that almost maximal growth again forces geometric consequences.  
\end{abstract}

\maketitle
\thispagestyle{empty}

\tableofcontents

\section{Introduction}

Let $S = K[x_1\dots,x_r]$, where $K$ is an algebraically closed field {of any characteristic}.  Let $J \subset S$ be a homogeneous ideal, so $J = \bigoplus_{t \in \mathbb Z}  [J]_t$.  The {\it Hilbert function} of $S/J$ is the numerical function defined by $h_{S/J}(t) =  \dim_K [S/J]_t$.   For any positive integer $n$, a theorem of Macaulay gives an upper bound for $h_{S/J}(n+1)$ in terms of $h_{S/J}(n)$ and $n$.  When $h_{S/J}(n+1)$ achieves this bound, we say that $S/J$ has {\it maximal growth} in degree $n$.  If $h_{S/J}(n+1)$ fails by 1 to {achieve} the bound, we say that $S/J$ has {\it almost maximal growth} in degree $n$.  

{In this paper we give results assuming almost maximal growth for
$S/J$ in degree $n$ to degree $n+1$.  We consider  two settings. First, the setting of arbitrary homogeneous ideals. Secondly, we will restrict ourselves to the artinian reductions of ideals of sets of points: if $Z \subset \mathbb P^r$ is a finite set of points, with homogeneous ideal $I_Z$, and $L$ is a general linear form, the ideal $J = \frac{\langle I_Z,L \rangle}{\langle L \rangle}\subset S$ is the {\it artinian reduction} of $I_Z$ by $L$.
}

For fixed $n$, the projectivization of the vector space $[J]_n$ is a linear system of hypersurfaces in $\mathbb P^{r-1}$.  As such, it may be basepoint-free or it may have a non-empty base locus.  {The starting observation of our research is that if $S/J$ has maximal growth, then a result of Gotzmann \cite{Gotzmann} forces} the existence of a non-empty base locus, and it gives the dimension and degree of this locus (as a scheme) through the Hilbert polynomial.  {For the artinian reduction of a set of points $Z$, in \cite{BGM} Bigatti, Geramita and the second author   apply Gotzmann's result}, giving careful information about {the} decomposition of $Z$ into the subset lying on the base locus and the subset off the base locus {(this is also a generalization of an old result of Davis for points in $\pp^2$, see \cite{davis})}.  These results are recalled in section \ref{prelim}.

In the case of almost maximal growth, Gotzmann's result no longer applies, and one focus of this paper is the fact that more possibilities arise than in the case of maximal growth.  {W}e give results that explain what kinds of base loci can occur, depending on the kind of almost maximal growth involved.  (This latter is a technicality arising from Macaulay's theorem and will be made precise in sections \ref{prelim} and \ref{arb ideals}.)  In almost all cases, a base locus is forced, but now it can have one of two possible dimensions, and different degrees; see for instance theorem \ref{n+1 choose n}, example \ref{different poss} and theorem \ref{other AMG}.  These results are given in section \ref{arb ideals}. 

In the only case where a base locus is not forced, it can still happen that a base locus exists.  This is the topic of section \ref{k k-1 BL}. In this case, turning to the {artinian reductions of ideals of sets of points}, we give results about a decomposition of $Z$, analogous to the one mentioned above.  

The last possibility is that there is no base locus.  This is dealt with in section \ref{k k-1 BPF}.  In this situation we find our most surprising result.  Despite the absence of a base  locus, it still turns out that many of the points are forced to lie on a plane. This is Theorem \ref{no base locus}.  It is very reminiscent of a result of Maroscia \cite{mar}, which says essentially that if the value of the Hilbert function of $S/J$ takes two or more consecutive values less than $r$ then many of the points must lie on a linear space of specified dimension.

There are several motivations for studying the geometry of sets of points, when the artinian reduction achieves almost maximal growth. We should mention that a similar analysis, in the case of sets of points with uniform position properties, allowed results for the theory of curves: bounds on the genus,  existence of special linear series, postulation of nodes of general plane projections (see \cite{CCDG}, \cite{CCi}, \cite{CCG}). Recent applications to the study of symmetric tensors (for which the uniform position properties may not hold) can be found in \cite{BB}, \cite{BGL}, \cite{BC}. 

At the end of the paper, we make a series of remarks on related problems and further extensions.


\section{Preliminaries}  \label{prelim}

Let $S = K[x_1,\dots,x_r]$, where $K$ is an algebraically closed field of {arbitrary characteristic}. 
 Let $J \subset S$ be a homogeneous ideal, so $J = \bigoplus_{t \in \mathbb Z}  [J]_t$.  

\begin{defn}  
The {\it Hilbert function} of $S/J$ is the numerical function defined by 
$h_{S/J}(t) =  \dim [S/J]_t$ for $t \geq 0$.
\end{defn} 

We begin by recalling  results of Macaulay and Green  that we will need in this paper.

\begin{defn} Let $k$ and $i$ be positive integers. The {\it i-binomial expansion of $k$} is
\[
k_{(i)} = \binom{k_i}{ i}+\binom{k_{i-1}}{i-1}+...+\binom{k_j}{j},
\]
 where $k_i>k_{i-1}>...>k_j \geq j \geq 1$. We remark that such an expansion always exists and it is unique (see, e.g.,  \cite{BH} Lemma 4.2.6).
\end{defn}

\begin{notn}
Following \cite{BG}, we define, for any integers $a$ and $b$,
\[
\left(k_{(i)}\right)_{a}^{b}=\binom{k_i+b}{i+a}+\binom{k_{i-1}+b}{i-1+a}+...+\binom{k_j+b}{j+a},
\]
where we set $\binom{m}{q}=0$ whenever $m<q$ or $q<0$.  Furthermore, we will set
\[ 
k^{\langle i \rangle} = \left(k_{(i)}\right)_{1}^{1} \hbox{ \ \ \ and \ \ \ } k_{\langle i \rangle} = \left(k_{(i)}\right)_{0}^{-1}.
\]
\end{notn}

\begin{thm}\label{gr}  Let $L \in [S]_1$ be a general linear form and let $A = S/J$ be a standard graded algebra. Denote by $h_n$  the degree $n$ entry of the Hilbert function of $A$ and by $h_n^{'}$ the degree $n$ entry of the Hilbert function of $A/L A$. Then:
\begin{itemize}
\item[(i)] (Macaulay) $\displaystyle h_{n+1}\leq h_n^{\langle n \rangle}.$

\item[(ii)]  (Green) $\displaystyle h_n^{'}\leq (h_n)_{\langle n \rangle}.$

\item[(iii)] (Gotzmann Persistence Theorem) Assume that $h_A(t)$ has maximal growth from degree $n$ to degree $n+1$ and that {$J$} has no minimal generator in degree $\geq n+2$.  Then
\[
h_A(n+d) =  \left ( \left(h_{A}(n) \right)_{(n)} \right )_{d}^{d}
\]
for all $d \geq 0$.    In particular, the Hilbert function equals the Hilbert polynomial in all degrees $\geq n$.
\end{itemize}
\end{thm}

\begin{proof}
(i) See \cite{BH}, Theorem 4.2.10.

(ii) See \cite{Green}, Theorem 1.

(iii) See \cite{Gotzmann}.
\end{proof}

\begin{defn}\label{amg}
If $h_A(n+1) = h_A(n)^{\langle n \rangle}$, we say that $h_A$ has {\it maximal growth} {in degree $n$.    We also sometimes say that $A$ has maximal growth from degree $n$ to degree $n+1$}.  If $h_A(n+1) = h_A(n)^{\langle n \rangle} -1$, we say that $h_A$  (or simply $A$) has {\it almost maximal growth}
{in degree $n$}.
\end{defn}

One can see from the result of Gotzmann that there are different kinds of maximal growth of the Hilbert function, depending on the degree of the corresponding Hilbert polynomial.

\begin{defn}
Fix  positive integers $n$ and $k$. Suppose that the $n$-binomial expansion of $k$ is
\[
k = \binom{k_n}{n} + \binom{k_{n-1}}{n-1} + \dots.
\]
Then the {\it MG-dimension} of $k$ in degree $n$, denoted $MG(k,n)$, is $k_n - n$.
\end{defn}

{Indeed, if $h_A(n) = \binom{k_n}{n} + (\hbox{lower degree terms})$, $h_A$ has maximal growth from degree $n$ to degree $n+1$ and $J$ has no generators in degree $\geq n+2$, then the Hilbert polynomial of $A$ has degree $k_n - n$, by Gotzmann's theorem.  This means that the base locus of the linear system is a scheme of dimension $k_n - n$ in $\mathbb P^{r-1}$.  

Notice that we reduce to the case where $J$ has no minimal generators in degree $\geq i$ if we substitute $J$ with its {\it truncation} $J_{\leq i}$, i.e. the ideal generated by the generators of $J$ which have degree $\leq i$. The ideal $J_{\leq i}$ defines the base locus of $[J_i]$, although it is not necessarily saturated.

Given any value $h_A(n)$, the above discussion shows that $MG(k,n)$ is the dimension for the base locus of $[J]_n$, in case $h_A(n+1)$ is the maximum allowed by Macaulay's theorem.
}

\smallskip
Let $R = K[x_0,\dots,x_r]$ and let $\mathbb Z \subset \mathbb P^n$ be a reduced, zero-dimensional scheme.  We will denote by $h_Z$ the Hilbert function of $R/I_Z$.   Let $L \in [R]_1$ be a general linear form.  The ideal $J = \frac{\langle I_Z,L \rangle}{\langle L \rangle} \subset R/\langle L \rangle \cong S$ is the {\it artinian reduction} of $R/I_Z$, and its Hilbert function is the {\it $h$-vector} of $Z$, or of $R/I_Z$.  It is possible for $R/I_Z$ to fail to have maximal growth from degree $n$ to degree $n+1$, but nevertheless $S/J$ does have maximal growth there (see \cite{BGM}  Example 1.3 (a)).  

The main idea of \cite{BGM} was that if $Z \subset \mathbb P^r$ is a reduced, zero-dimensional scheme with artinian reduction $J$, such that $S/J$ has maximal growth in degree $n$ (even if it is not the case that $R/I_Z$  has maximal growth there), this still has strong implications for the geometry of $Z$.  The proofs in \cite{BGM} heavily used the Gotzmann Persistence Theorem.  

Our interest in this paper focuses on the situation where the growth of the first difference of the Hilbert function is {\it almost maximal} as defined in Definition \ref{amg}.  As mentioned above, there are different kinds of maximal growth, so correspondingly there are different cases of almost maximal growth to be analyzed.  
Although this paper should be viewed as a generalization of parts of \cite{BGM}, it should also be noted that \cite{BGM} was motivated by the following theorem of E.D. Davis, which we state more in the language of this paper.  (He actually proved a bit more, including the fact that he did not assume that $Z$ is reduced).

\begin{thm}[see \cite{davis} Theorem 4.1 and Corollary 4.2]

Let $Z \subset \mathbb P^2$ be a reduced, finite set of points with homogeneous ideal $I_Z$.  Assume that $\Delta h_Z(n) = \Delta h_Z(n+1) = k$.  Then $[I_Z]_n$ and $[I_Z]_{n+1}$ have a GCD, say $F$, of degree $k$.  Furthermore, $F$ {defines a  reduced curve} and 

\begin{itemize}

\item[(a)] $(I_Z : F)$ is the saturated ideal of the subset, $Z_2$, of $Z$ not lying on the curve defined by~$F$.  

\item[(b)] $(I_Z,F)$ is the saturated ideal of the subset of $Z_1$ lying on $F$.

\item[(c)] $\Delta h_Z(t) = \Delta h_{Z_2}(t-k) + h_{Z_1}(t)$ for all $t$.

\item[(d)] $\Delta h_{Z_2}(t) = 0$ for $t \geq n-k$.

\end{itemize}

\end{thm}

\begin{rem} \label{GLP remark}
{A scheme $X$ with ideal sheaf $\mathcal I_X$ is $t$-regular if $H^i(\mathcal I_X(t-i)=0$
for $i>0$.}

As in \cite{BGM}, we will need the following results from \cite{GLP}.  

\begin{itemize}
\item[(a)] (\cite{GLP} Theorem 1.1) If $X \subset \mathbb P^n$ is a reduced irreducible non-degenerate curve of degree $d$ then $\mathcal I_X$ is $(d+2-n)$-regular.

\item[(b)] (\cite{GLP} Remark (1), page 497) Let $X \subset \mathbb P^n$ be a reduced but not necessarily irreducible curve.  Suppose $X$ has irreducible components $X_i$ of degree $d_i$, and that $X_i$ spans a $\mathbb P^{n_i} \subset \mathbb P^n$.  Set
\[
m_i = 
\left \{
\begin{array}{ll}
d_i + 2 - n_i, & \hbox{if } d_i \geq 2; \\
1 & \hbox{if } d_i = 1 \hbox{ (i.e., if $X_i$ is a line)}
\end{array}
\right.
\]
Then $X$ is $(\sum m_i)$-regular.
\qed
\end{itemize}

\end{rem}


\section{Forcing a base locus} \label{arb ideals}

In this section we investigate when almost maximal growth of the Hilbert function of a standard graded algebra $S/J$ from degree $n$ to degree $n+1$ forces the existence  of a base locus in $[J]_n$.  We do not yet assume that $J$ is the artinian reduction of the ideal of a set of points in projective space.

The next definition merely gives a notation for the dimension of the base locus of the component of an ideal in a degree $n$ under the assumption that the quotient has  maximal growth (rather than almost maximal growth) from degree $n$ to degree $n+1$.

Let $J \subset S = k[x_1,\dots,x_r]$ be a homogeneous ideal and let $A = S/J$.  Suppose that $h_A$ has maximal growth from degree $n$ to degree $n+1$.  Then the base locus of $[J]_n$ and the base locus of $[J]_{n+1}$, viewed in $\mathbb P^{r-1}$, coincide and this locus is a scheme of dimension $MG(h_A(n),n)$.  If $J$ is the artinian reduction of the homogeneous ideal of a zero-dimensional scheme $Z$ by a general linear form, then $[I_Z]_n$ has base locus of dimension $MG(h_{A}(n),n)+1$, and the same holds for $[I_Z]_{n+1}$.  This is the idea behind \cite{BGM}.  We now extend this idea to our setting. 

The first step is the case where the MG-dimension of $S/J$ is 1.  We assume that the growth of the Hilbert function of $S/J$ is one less than maximal, and we will show that the dimension of the base locus is either the same as that occurring for maximal growth, or one less.
In Example \ref{different poss} we show that both possibilities for the dimension of the base locus can occur, and we indicate some possibilities for the degree of the base locus (recalling that for maximal growth the degree and dimension are both forced).

\begin{thm} \label{n+1 choose n}

Let $J \subset S$ be a homogeneous ideal, and let $A = S/J$.  Assume that 
\begin{equation} \label{assumption-curve}
h_A(n) = \binom{n+1}{n} + \hbox{(lower terms)}
\end{equation}
is the $n$-binomial expansion of $h_A(n)$. Equivalently, assume that $n+1 \leq h_A(n) < \binom{n+2}{2}$.  Assume further that $h_{A}(n+1) =   h_A(n)^{\langle n \rangle} -1$, i.e. that the growth is one less than maximal.  Then the base locus of $[J]_n$ in $\mathbb P^{r-1}$ {exists and it} is a scheme of dimension either 0 or 1. In other words, the dimension of the base locus is either $MG(h_A(n),n) -1$ or $MG(h_A(n),n)$.
\end{thm}

\begin{proof}

For any homogeneous ideal $I \subset S$ and linear form $L$ we have the exact sequence

\begin{equation} \label{ses}
\begin{tikzpicture} [baseline=(current  bounding  box.center)]
\node (v1) at (1,0) {$0$};
\node (v3) at (3,0) {$\displaystyle \frac{I:L}{I}(-1)$};
\node (v5) at (5.7,0) {$S/I(-1)$};
\node (v7) at (8.3,0) {$S/I$};
\node (v9) at (10.4,0) {$S/(I,L)$};
\node (v11) at (12.25,0) {$0$};
\node (v12) at (7,-1.2) {$S/(I:L)(-1)$};
\node (v13) at (6,-2.2) {$0$};
\node (v14) at (8,-2.2) {$0$};
\node (v15) at (7.2,0.35) {\scriptsize $\times L$};
\draw [->] (v1) -- (v3);
\draw [->] (v3) -- (v5);
\draw [->] (v5) -- (v7);
\draw [->] (v7) -- (v9);
\draw [->] (v9) -- (v11);
\draw [->] (v5) -- (v12);
\draw [->] (v12) -- (v7);
\draw [->] (v13) -- (v12);
\draw [->] (v12) -- (v14);
\end{tikzpicture}
\end{equation}

We shall use this sequence repeatedly, especially the second short exact sequence.  Notice that $S/(I:L)(-1)$ is the image of $S/I(-1)$ in $S/I$ under multiplication by $L$.

We first note that under our hypotheses, the base locus cannot have dimension $>1$.  Indeed, the Hilbert polynomial of $S/J_{\leq n}$ has degree at most 1.  We {just} have  to show that the base locus is not empty.

Let $\ell$ be a general linear form.  For convenience let us denote by
\[
\begin{array}{rcl}
k & = & h_A(n) \\
p & = & \dim h_{S/(J,\ell)}(n) \\
s & = & \dim h_{S/(J,\ell)}(n+1)
\end{array}
\]
Thus we have the table 

\bigskip

\begin{tabular}{r|ccccccccccc}
degree $i$ & 0 & 1 & $\dots$ & $n$ & $n+1$     \\  \hline
$\dim [S/J]_i$ & 1 & $r$ & $\dots$ & $k$ & $k^{\langle n \rangle} -1$    \\
$\dim [S/(J:\ell)]_{i-1}$ &  & 1 & \dots & $k-p$ & $k^{\langle n \rangle} - 1 -s$   \\
$\dim [S/(J,\ell)]_i$ &  1 & $r-1$ & $\dots$ & $p$ & $s$
\end{tabular}

\bigskip

There exist $q$ and $m$ giving the following binomial expansions:
\[
\begin{array}{rcl}
\displaystyle k & = & \displaystyle \underbrace{\binom{n+1}{n} + \binom{n}{n-1} + \dots + \binom{n-q+2}{n-q+1}}_{\hbox{$q$ terms}} + \underbrace{\binom{n-q}{n-q} + \dots + \binom{n-q-m+1}{n-q-m+1}}_{\hbox{$m$ terms}} \\
&= & \displaystyle \binom{n+1}{n} + \binom{n}{n-1} + \dots + \binom{n-q+2}{n-q+1} + m
\end{array}
\]
and
\[
k^{\langle n \rangle} = \binom{n+2}{n+1} + \binom{n+1}{n} + \dots + \binom{n-q+3}{n-q+2} + m.
\]
By Green's theorem, and then by Macaulay's theorem we have
\[
p \leq q \leq n \ \ \ \hbox{ and } \ \ \ s \leq p \leq q.  
\]
Notice that 
\[
m \leq n-q, \ \ \hbox{ and in particular if $n=q$ then $m=0$.}
\]  
We have from the condition (\ref{assumption-curve})  that $n < h_A(n) = k  < \binom{n+2}{2} = \binom{n+2}{n}$.  

From the exact sequence (\ref{ses}) we also know that 
\[
k^{\langle n \rangle} - 1 -s \leq k
\]
so, by the properties of binomials,
\[
\begin{array}{rcl}
s & \geq & k^{\langle n \rangle} -k-1 \\
& = & \displaystyle \binom{n+2}{n+1} - \binom{n-q+2}{n-q+1} - 1 \\
& = & q-1.
\end{array}
\]
All together we have
\[
q-1 \leq s \leq p \leq q \leq n.
\]
This means that the possibilities for $(p,s)$ are $(q-1,q-1)$, $(q,q-1)$ or $(q,q)$.  The first and third of these represent maximal growth for $S/(J,\ell)$ since $q \leq n$, so the base locus for $[(J,\ell)]_n$ is zero-dimensional, by Theorem \ref{gr} (iii). Thus (since $\ell$ was general) the base locus for $[J]_n$ has dimension $1 = MG(h_A(n),n)$.

{Assume now} $p=q$, $s = q-1$.    This means
\[
\dim [S/(J:\ell)]_{n-1} = k-q, \ \ \hbox{ and } \ \ \dim [S/(J:\ell)]_n = k^{\langle n \rangle} - q,
\]
i.e. we have the table

\bigskip

\begin{tabular}{r|ccccccccccc}
degree $i$ & 0 & 1 & $\dots$ & $n$ & $n+1$     \\  \hline
$\dim [S/J]_i$ & 1 & $r$ & $\dots$ & $k$ & $k^{\langle n \rangle} -1$    \\
$\dim [S/(J:\ell)]_{i-1}$ &  & 1 & \dots & $k-q$ & $k^{\langle n \rangle} - q$   \\
$\dim [S/(J,\ell)]_i$ &  1 & $r-1$ & $\dots$ & $q$ & $q-1$
\end{tabular}

\bigskip

(Notice that the passage from $k$ to $k^{\langle n \rangle}-1$ occurs from degree $n$ to degree $n+1$, while the growth from $k-q$ to $k^{\langle n \rangle} -q$ is from degree $n-1$ to degree $n$.)  We have
\[
\begin{array}{rcl}
k-q & = & \displaystyle  h_A(n) - q = \left [ \binom{n}{n-1} + \dots + \binom{n-q+2}{n-q+1} \right ] + \binom{n+1}{n} + m - q \\
& = &  \displaystyle  \left [ \binom{n}{n-1} + \dots + \binom{n-q+2}{n-q+1} \right ] + n + 1 + m - q 
\end{array}
\]
and {similarly}
\[
\begin{array}{rcl}
k^{\langle n \rangle} - q & =   & \displaystyle \left [ \binom{n+1}{n} + \dots + \binom{n-q+3}{n-q+2} \right ] + n+2+m-q.
\end{array}
\]

We have the following cases.

\medskip

\begin{itemize}


\item[$\blacktriangleright$ Case 1.]  If $n=q$ (hence $m=0$) then 
\[
k-q = \binom{n}{n-1} + \dots + \binom{2}{1} + 1
\]
and
\[
k^{\langle n \rangle} - q  =   \binom{n+1}{n} + \dots + \binom{3}{2} + 2
\]
which exceeds maximal growth from degree $n-1$ to degree $n$, so this case cannot occur.

\medskip

\item[$\blacktriangleright$ Case 2.] If $n > q$ and $m < n-q$ then we get {
$\binom{n+1}n -q=\binom{n-q+1}{n-q}$, so that}
{\footnotesize
\[
\begin{array}{rcl}
k-q & = & \displaystyle \left [ \binom{n}{n-1} + \dots + \binom{n-q+2}{n-q+1} + \binom{n-q+1}{n-q} \right ] +m \\ \\
& = & \displaystyle  \left [ \binom{n}{n-1} + \dots + \binom{n-q+2}{n-q+1} + \binom{n-q+1}{n-q} \right ] + \left [\underbrace{\binom{n-q-1}{n-q-1} + \binom{n-q-2}{n-q-2} + \dots }_{m \hbox{ \tiny times}} \right ] 
\end{array}
\] }
and {similarly}
{\footnotesize 
\[
\begin{array}{rcl}
k^{\langle n \rangle} - q & =   & \displaystyle   \binom{n+1}{n} + \dots + \binom{n-q+3}{n-q+2} + \binom{n-q+2}{n-q+1}   + m \\ \\
& = & \displaystyle  \left [ \binom{n+1}{n} + \dots + \binom{n-q+3}{n-q+2} + \binom{n-q+2}{n-q+1} \right ] + \left [\underbrace{\binom{n-q}{n-q} + \binom{n-q-1}{n-q-1} + \dots }_{m \hbox{ \tiny times}} \right ] 
\end{array}
\] }
which is clearly maximal growth.  In this case we have that the base locus of $[J:\ell]_{n-1}$ and of $[J:\ell]_n$ has dimension 1, so the base locus of $[J]_n$ has dimension at least $1 = MG(h_A(n),n)$ since $[J]_n \subset [J:\ell]_n$.

\medskip

\item[$\blacktriangleright$ Case 3.] If $n > q$ and $m = n-q$ then {we have as above}
\[
\begin{array}{rcl}
k-q & = & \displaystyle  \binom{n}{n-1} + \dots + \binom{n-q+2}{n-q+1} + \binom{n-q+1}{n-q} +m
\end{array}
\]
and
\[
\begin{array}{rcl}
k^{\langle n \rangle} - q & =   & \displaystyle \left [ \binom{n+1}{n} + \dots + \binom{n-q+3}{n-q+2} + \binom{n-q+2}{n-q+1} \right ] + m
\end{array}
\]
{but now $m> \binom{n-q-1}{n-q-1} + \dots +\binom 1 1$, so we cannot conclude as in the previous case.

Instead, we consider}  subcases as follows.

\medskip

\begin{itemize}

\item[$\bullet$ 3A.] If $m = n-q = 1$, let us first summarize our assumptions at this point.
\[
\begin{array}{rcl}
k & = & \displaystyle \left [ \binom{n+1}{n} + \dots + \binom{3}{2} \right ] + \binom{1}{1} \\
k^{\langle n \rangle} & = & \displaystyle \left [ \binom{n+2}{n+1} + \dots + \binom{4}{3} \right ] + \binom{2}{2} \\
p & =  &  q \ \   = \ \   n-1 \\
s & = &  q-1 \ \ = \ \ n-2 \\
k-q & = & \displaystyle \binom{n}{n-1} + \dots + \binom{3}{2} + \binom{2}{1} + 1 = \binom{n+1}{n-1} \\
k^{\langle n \rangle} - q & = & \displaystyle  \left [ \binom{n+1}{n} + \dots + \binom{3}{2} \right ] + \binom{1}{1}
\end{array}
\]
Notice that the growth from $k-q$ in degree $n-1$ to $k^{\langle n \rangle} -q$ in degree $n$ in this case is two less than maximal growth.  A key observation is that we have 
\[
k = k^{\langle n \rangle} - q;
\]
that is, 
\[
\dim [S/J]_n = \dim [S/(J:\ell)]_n.
\]
This means that the multiplication $\times \ell : [S/J]_n \rightarrow [S/J]_{n+1}$ is injective.  Let $J_1 = (J : \ell)$.  Let $\ell_2$ be a general linear form.  Since $J$ does not depend on $\ell$, multiplication by $\ell_2$ is also injective. Then from the commutative diagram
\begin{equation} \label{snake lem}
\begin{array}{ccccccccccccccccc}
&&&& 0  \\
&&&& \downarrow \\
0 & \rightarrow & [S/J_1]_{n-1} & \stackrel{\times \ell}{\longrightarrow} & [S/J]_n & \rightarrow & [S/(J,\ell)]_n & \rightarrow & 0 \\
&& \phantom{\scriptstyle \times \ell_2} \downarrow {\scriptstyle \times \ell_2} && \phantom{\scriptstyle \times \ell_2} \downarrow {\scriptstyle \times \ell_2} && \phantom{\scriptstyle \times \ell_2} \downarrow {\scriptstyle \times \ell_2} \\
0 & \rightarrow & [S/J_1]_{n} & \stackrel{\times \ell}{\longrightarrow} & [S/J]_{n+1} & \rightarrow & [S/(J,\ell)]_{n+1} & \rightarrow & 0 \\
\end{array}
\end{equation}
and {by} the Snake Lemma, we see that $\times \ell_2 : [S/J_1]_{n-1} \rightarrow [S/J_1]_n$ is also injective.  We obtain
\begin{equation} \label{equality}
\dim [S/J_1]_{n-1} = \dim [S/(J_1 : \ell_2)]_{n-1}.
\end{equation}
Making a similar analysis but using $J_1$ in place of $J$, we have, after a short calculation, 

\bigskip

\begin{tabular}{r|ccccccccccc}
degree $i$ & 0 & 1 & $\dots$ & $n-1$ & $n$     \\  \hline
$\dim [S/J_1]_i$ & 1 & $r$ & $\dots$ & $\binom{n+1}{n-1}$ & $\left [ \binom{n+1}{n} + \dots + \binom{3}{2} \right ] + \binom{1}{1}$    \\
$\dim [S/(J_1:\ell_2)]_{i-1}$ &  & 1 & \dots & $\binom{n+1}{n-1}-p_1$ & $\binom{n+1}{n-1}$   \\
$\dim [S/(J_1,\ell_2)]_i$ &  1 & $r-1$ & $\dots$ & $p_1$ & $n-1$
\end{tabular}

\bigskip

for some $p_1$.  Combining Green's theorem (in degree $n-1$) and Macaulay's theorem on the bottom line, we obtain
\[
n-1 \leq p_1 \leq n+1.
\]
On the other hand, from the second line and again invoking {Macaulay's theorem} (remembering the shift), we have
\[
\binom{n+1}{n-1} - p_1 \geq \binom{n}{n-2}
\]
(since $\binom{n}{n-2}$ is the smallest value in degree $n-2$ that could grow to $\binom{n+1}{n-1}$ in degree $n-1$), or
\[
p_1 \leq n.
\]
Thus $p_1$ is either equal to $n-1$ or $n$. 

\medskip

\begin{itemize}

\item[*] If $p_1 = n$ then the second line of the last table gives 
\[  
\dim [S/(J_1 : \ell_2)]_{n-2} = \binom{n}{n-2},
\]
so this represents maximal growth for $S/(J_1 : \ell_2)$ from degree $n-2$ to degree $n-1$.  From the corresponding Hilbert polynomial we see that $[(J_1 : \ell_2)]_{n-2}$ and $[(J_1 : \ell_2)]_{n-1}$ are the degree $n-2$ and $n-1$ components of the saturated ideal of a linear space $\Lambda$Πof dimension 2.  
We have
\[
[I_\Lambda]_{n-1} = [(J_1 : \ell_2)]_{n-1} = [J_1]_{n-1} = [(J : \ell)]_{n-1}.
\]
On the other hand, we have observed that $\dim [S/(J:\ell)]_{n}$ is two less than the maximum possible.  Thus $(J:\ell)$ picks up two minimal generators in degree $n$, so the base locus of $[J:\ell]_n$ is at least zero-dimensional.  Since we have also seen in this case that $[J]_n = [J:\ell]_n$, we see that the base locus of $[J]_n$ is at least zero dimensional, so we have our desired result.

\bigskip

\item[*] If $p_1 = n-1$, then $S/(J_1,\ell_2)$ has maximal growth from degree $n-1$ to degree $n$.  This means that $([J_1, \ell_2)]_{n-1}$ and $[(J_1,\ell_2)]_n$ have a base locus consisting of a zero-dimensional scheme.  Thus $[J_1]_{n-1}$ and $[J_1]_n$ have a base locus consisting of a 1-dimensional scheme, since $\ell_2$ is general.  In particular, $[(J:\ell)]_n$ has a 1-dimensional base locus.  Since $[J]_n \subset [(J:\ell)]_n$, we have the desired result.

\end{itemize}

\bigskip

\item[$\bullet$ 3B.] If $m = n-q > 1$ then as before we summarize our current assumptions.
\[
\begin{array}{rcl}
k & = & \displaystyle \left [ \binom{n+1}{n} + \dots + \binom{n-q+2}{n-q+1} \right ] + \left [ \binom{n-q}{n-q} + \dots + \binom{1}{1} \right ] \\
k^{\langle n \rangle} & = & \displaystyle \left [ \binom{n+2}{n+1} + \dots + \binom{n-q+3}{n-q+2} \right ] + \left [ \binom{n-q+1}{n-q+1} + \dots + \binom{2}{2} \right ] \\
p & =  &  q \  = \ s+1 \\
k-q & = & \displaystyle \left [ \binom{n}{n-1} + \dots + \binom{n-q}{n-q-1} \right ] \\
k^{\langle n \rangle} - q & = & \displaystyle  \left [ \binom{n+1}{n} + \dots + \binom{n-q+2}{n-q+1} \right ] + \left [ \binom{n-q}{n-q} + \dots + \binom{1}{1} \right ]
\end{array}
\]
Notice that again, we have 
\[
k = k^{\langle n \rangle} - q, \ \ \hbox{ i.e. } \ \ \dim[S/J]_n = \dim [S/(J:\ell)]_n,
\]
so $[S/J]_n \stackrel{\times \ell}{\longrightarrow} [S/J]_{n+1}$ is injective.  We again let $J_1 = (J: \ell)$ and $\ell_2$ a general linear form.  Again using the Snake Lemma, we obtain that $[S/J_1]_{n-1} \stackrel{\times \ell_2}{\longrightarrow} [S/J_1]_{n}$ is injective, so
\[
\dim [S/J_1]_{n-1} = \dim [S/(J_1 : \ell_2)]_{n-1}
\]
Let $p_1 = \dim [S/(J_1,\ell_2)]_{n-1}$ and $p_2 = \dim [S/(J_1,\ell_2)]_n$.  We compute
\[
\begin{array}{rcl}
p_2 & = & \dim [S/J_1]_n - \dim S/(J_1 : \ell_2)]_{n-1} \\
& = & \left [ \binom{n+1}{n} + \dots + \binom{n-q+2}{n-q+1} \right ] + \left [ \binom{n-q}{n-q} + \dots + \binom{1}{1} \right ] \\
&& - \left [ \binom{n}{n-1} + \dots + \binom{n-q}{n-q-1} \right ] \\
& = & q.
\end{array}
\]
We thus have the table

\bigskip

\begin{tabular}{r|ccccccccccc}
degree $i$ & \dots &  $n-1$ & $n$     \\  \hline
$\dim [S/J_1]_i$ &  $\dots$ & $\binom{n}{n-1} + \dots + \binom{n-q}{n-q-1}$ & $\left [ \binom{n+1}{n} + \dots + \binom{n-q+2}{n-q+1} \right ]$ \\
&&& $+ \left [ \binom{n-q}{n-q} + \dots + \binom{1}{1} \right ]$    \\
$\dim [S/(J_1:\ell_2)]_{i-1}$ &  \dots & $\binom{n}{n-1} + \dots + \binom{n-q}{n-q-1} - p_1$ & $\binom{n}{n-1} + \dots + \binom{n-q}{n-q-1}$   \\
$\dim [S/(J_1,\ell_2)]_i$  & $\dots$ & $p_1$ & $q$
\end{tabular}

\bigskip

From Green's theorem, {the first and the third  lines of the table show that} $p_1 \leq q+1$. From one of the assumptions in Case 3 we have $q < n$, so from Macaulay's theorem applied to $S/(J_1,\ell_2)$ we obtain $p_1 \geq q$.  Combining, we have
\[
q \leq p_1 \leq q+1.
\]

We have the following possibilities.

\medskip

\begin{itemize}

\item[*] If $p_1 = q$ then the bottom line of the above table represents maximal growth from degree $n-1$ to degree $n$.  This means that the base loci of $[(J_1,\ell_2)]_{n-1}$ and $[(J_1,\ell_2)]_n$ have dimension 0.  Hence $[J_1]_{n-1}$ and $[J_1]_n$ have a 1-dimensional base locus.  In particular, $[(J:\ell)]_n$ has a 1-dimensional base locus.  Since $[J]_n \subset [(J:\ell)]_n$, we have the desired result.

\bigskip

\item[*]  Assume that $p_1 = q+1$ and $n-q-1 > 1$.   We note that since $p_1 = q+1$,  and since $n-q-1 > 1$, we have 
\[
\begin{array}{rcl}
\dim [S/(J_1 : \ell_2)]_{n-2} & = & \displaystyle  \binom{n-1}{n-2} + \dots + \binom{n-q}{n-q-1} + n-q-1. \\ \\
& = & \displaystyle  \binom{n-1}{n-2} + \dots + \binom{n-q}{n-q-1} + \binom{n-q-1}{n-q-2}, 
\end{array}
\]
so we have maximal growth for $S/(J_1 : \ell_2)$ from degree $n-2$ to degree $n-1$, giving a 1-dimensional base locus.  The argument is then the same as above.  

\bigskip

\item[*] Assume finally that $p_1 = q+1$ and $n-q-1 = 1$.  {After a short computation, one sees that} 
\[
\dim [S/(J_1 : \ell_2)]_{n-2} = \binom{n}{n-2}.
\]   
{So now we} have the table

\bigskip

\begin{tabular}{r|ccccccccccc}
degree $i$ & \dots &  $n-1$ & $n$     \\  \hline
$\dim [S/J_1]_i$ &  $\dots$ & $\binom{n}{n-1} + \dots + \binom{2}{1}$ & $\left [ \binom{n+1}{n} + \dots + \binom{4}{3} \right ] +2$ \\
$\dim [S/(J_1:\ell_2)]_{i-1}$ &  \dots & $\binom{n}{n-2} $ & $\binom{n}{n-1} + \dots + \binom{2}{1}$   \\
$\dim [S/(J_1,\ell_2)]_i$  & $\dots$ & $q+1$ & $q$
\end{tabular}

\bigskip

The exact sequence (\ref{ses}) splits into two short exact sequences as indicated (but use $I = J_1$ and $L = \ell_2$).  The fact that 
\[
\dim [S/J_1]_{n-1} = \dim [S/(J_1 : \ell_2)]_{n-1}
\] 
means that $[\frac{J_1:\ell_2}{J_1}]_{n-1} = 0$, so from the long exact sequence (\ref{ses}) we see that 
$
\times \ell_2 : [S/J_1]_{n-1} \stackrel{\times \ell_2}{\longrightarrow} [S/J_1]_n
$
is an injection.  Thus for a general linear form $\ell_3$, $[S/J_1]_{n-1} \stackrel{\times \ell_3}{\longrightarrow} [S/J_1]_n$ is also injective.  Setting  $J_2 = (J_1 : \ell_2)$, we again obtain by the Snake Lemma (as we did in (\ref{snake lem}))  an injection   $[S/(J_1 : \ell_2)]_{n-2} \stackrel{\times \ell_3}{\longrightarrow} [S/(J_1 : \ell_2)]_{n-1}$.   We now have the following table.

\bigskip

\begin{tabular}{r|ccccccccccc}
degree $i$ & \dots &  $n-2$ & $n-1$     \\  \hline
$\dim [S/J_2]_i$ &  $\dots$ & $\binom{n}{n-2}$ & $\binom{n}{n-1} + \dots + \binom{2}{1}$ \\
$\dim [S/(J_2:\ell_3)]_{i-1}$ &  \dots & $\binom{n}{n-2} - p_3 $ & $\binom{n}{n-2}$   \\
$\dim [S/(J_2,\ell_3)]_i$  & $\dots$ & $p_3$ & $n-1$
\end{tabular}

\bigskip

where the bottom right dimension is the result of a simple computation.  Now combining Green's theorem and Macaulay's theorem, we obtain $n-1 \leq p_3 \leq n-1$, i.e. $p_3 = n-1$.  But then
\[
\dim [S/(J_2 : \ell_3)]_{n-3} = \binom{n}{n-2} - (n-1) = \binom{n-1}{n-3}.
\]
This means that $S/(J_2 : \ell_3)$ has maximal growth from degree $n-3$ to $n-2$, so its base locus is 2-dimensional.  In fact, {by Gotzmann's theorem, we can compute the Hilbert polynomial and obtain that the} base locus is a plane $\Lambda$, and $[J_2 : \ell_3]_{n-2} = [I_{\Lambda}]_{n-2}$. From the above table we then obtain that $[J_2]_{n-2} = [I_\Lambda]_{n-2}$.  Since the growth of the Hilbert function of $J_2$ from degree $n-2$ to degree $n-1$ is one less than maximal, the base locus of $[J_2]_{n-1}$ is (precisely) 1-dimensional.  Indeed, it must be the degree $n-1$ component of a plane curve of degree $n-1$.  From the previous table we have (recalling $J_2 = J_1 : \ell_2$) that $[J_2]_{n-1} = [J_1]_{n-1}$.  The growth of $S/J_1$ from degree $n-1$ to degree $n$ is again one less than maximal, so $[J_1]_{n}$ has at least a 0-dimensional base locus.  In the same way we obtain that $[J_1]_{n} = [J]_n$, so we obtain the desired result.

\end{itemize}
\end{itemize}
\end{itemize}

This completes the proof.
\end{proof}

\begin{rem} When the base locus has dimension 1,  its degree is at most the number $d$
of terms of the form $\binom{a+1}{a}$ in the $n$-binomial expansion of $h_A(n)$.  

Examples exist where this bound is achieved. 
This can immediately be seen by choosing an ideal $J_1$ for which $S/J_1$ has Hilbert function $h_A(n)$ in degree $n$ and has maximal growth to degree $n+1$.  Then form the ideal $J$ by adding one minimal generator to $J_1$ (say a general form) in degree $n+1$. 

The next example shows that all degrees $\leq d$ can occur at least for some values of $h_A(n)$. 
\end{rem}

\begin{ex} \label{different poss} 
In this example we show that all the possibilities left open by Theorem \ref{n+1 choose n} can occur.
{The computation were performed with the aid of the computer Algebra Package CoCoA (\cite{cocoa})}. 

Let $S = k[x,y,z]$, $n = 6$ and 
\[
h_A(6) = 21 = \binom{7}{6} + \binom{6}{5} + \binom{5}{4} + \binom{3}{3} + \binom{2}{2} + \binom{1}{1}.
\]
Then maximal growth would correspond to $h_A(7) = 24$, so the assumption of Theorem \ref{n+1 choose n} is that $h_A(7) = 23$.  Notice that a value of 21 in degree 6 allows a base locus of at most a curve of degree 3. Let 
\[
\begin{array}{rcl}
I & = & \langle x^6, x^5y, x^5z, x^4y^2, x^4yz, x^4z^2 \rangle \\
J_1 & = & \langle I, x^3y^3, x^3y^2z^2 \rangle \\

J_2 & = & \langle I, x^2y^4\rangle \\

J_3 & = & \langle I, xy^5 \rangle \\

J_4 & = & \langle F,G \rangle : I_P

\end{array}
\]
where $P$ is a general point in $\mathbb P^2$ and $\langle F,G \rangle$ is a general complete intersection of type $(4,6)$ in $I_P$.
Notice that $I$ and $J_1$ are lex-segment ideals, and $J_1$ has a minimal generator in degree 7 while none of $J_2$, $J_3$ or $J_4$ do.  ($J_4$ has generators in degrees 4, 6 and 8.)  
One can check that $S/J_1$, $S/J_2$, $S/J_3$ and $S/J_4$ all satisfy the assumption $h_A(6) = 21$ and $h_A(7) = 23$.  The Hilbert polynomials are as follows:
\[
\begin{array}{rcl}
HP(S/J_1) & = & 3t+2 \  \hbox{ for } \ t \geq 7; \\
HP(S/J_2) & = & 2t+9 \ \hbox{ for } \ t \geq 6; \\
HP(S/J_3) & = & t+16 \ \hbox{ for } \ t \geq 7; \\
HP(S/J_4) & = & 23 \ \hbox{ for } \ t \geq 7;
\end{array}
\]
Hence the base locus of $[J_1]_6$ and $[J_1]_7$ is a curve of degree 3 (defined by $x^3$), the base locus of $[J_2]_6$ and $[J_2]_7$ is a curve of degree 2 (defined by $x^2$), the base locus of $[J_3]_6$ and $[J_3]_7$ is a curve of degree 1 (defined by $x$) and the base locus of $[J_4]_6$ and $[J_4]_7$ is a zero-dimensional scheme of degree 24 defined by the complete intersection $\langle F,G \rangle$ (since the last generator reduces the degree by 1, thanks to a standard liaison computation).

\end{ex}

{When the value of the Hilbert function $h_A(n)$ is bigger than or equal to $\binom {n+2}2$, we have a similar behavior.}

\begin{thm} \label{other AMG}

Let $J \subset S$ be a homogeneous ideal, and let $A = S/J$.  Assume that $h_A(n) \geq \binom{n+2}{2}$.  Assume further that $h_{A}(n+1) = h_A(n)^{\langle n \rangle} -1$, i.e. that the growth is one less than maximal.  Then the base locus of $[J]_n$ in $\mathbb P^{r-1}$ is a scheme of dimension either $MG(h_A(n),n)-1$ or $MG(h_A(n),n)$.  

\end{thm}

\begin{proof}

We will use induction on $r$, the number of variables in $S$.    If $r=3$ then $[A]_{\leq n} = [S]_{\leq n}$ by the assumption that $h_A(n) \geq \binom{n+2}{2}$ (in fact equality must hold), so the base locus is all of $\mathbb P^2$.  Notice that in this case $[J]_{n+1}$ has a base locus consisting of a curve of degree $n+1$, by the assumption of almost maximal growth.

Let us write the $n$-binomial expansion of $h_A(n)$:
\begin{equation} \label{n-binom exp of hn}
h_A(n) = \binom{k_n}{n} + \binom{k_{n-1}}{n-1} + \dots + \binom{k_i}{i}
\end{equation}
where $k_n > k_{n-1} > \dots > k_i \geq i \geq 1$. The condition that $h_A(n) \geq \binom{n+2}{n}$ means that  $k_n \geq n+2$.  Hence $h_A(n)_{\langle n \rangle} \geq n+1$.

We have
\[
h_A(n)^{\langle n \rangle} =  \binom{k_n+1}{n+1} + \binom{k_{n-1}+1}{n} + \dots + \binom{k_i+1}{i+1},
\]
\begin{equation} \label{hAn1}
h_A(n)_{\langle n \rangle} = \binom{k_n-1}{n} + \binom{k_{n-1}-1}{n-1} + \dots + \binom{k_i-1}{i}
\end{equation}
and
\begin{equation} \label{hAn}
[h_A(n)_{\langle n \rangle}]^{\langle n \rangle} = \binom{k_n}{n+1} + \binom{k_{n-1}}{n} + \dots + \binom{k_i}{i+1}
\end{equation}
where some of the {binomials} may be zero in the last two equalities.  Indeed, all the terms are non-zero if and only if $k_i > i$.

Let $h$ be a general linear form in $S$.  Mimicking a proof in \cite{BH}, page 172, we have the following sequence of inequalities and equalities:
\begin{equation} \label{eqns}
\begin{array}{rclll}
[h_A(n)_{\langle n \rangle}]^{\langle n \rangle} & \geq & h_{A/hA}(n)^{\langle n \rangle} && \hbox{(Green's theorem)} \\
& \geq & h_{A/hA}(n+1) && \hbox{(Macaulay's theorem)} \\
& \geq & h_A(n+1) - h_A(n) \\
& = & h_A(n)^{\langle n \rangle} - 1 - h_A(n) && \hbox{(hypothesis}) \\
& = & [h_A(n)_{\langle n \rangle}]^{\langle n \rangle} -1.
\end{array}
\end{equation}
The third line comes from the exact sequence
\begin{equation} \label{exact seq}
A(-1) \stackrel{\times h}{\longrightarrow} A \rightarrow A/hA \rightarrow 0
\end{equation}
and the last line is a straightforward computation from the definitions.  This means that one of the inequalities must be strict (differing by 1) and the others must be equalities. 

Suppose that the first line of (\ref{eqns}) is a strict inequality and the rest are equalities, i.e. we have 
\begin{equation} \label{1st line}
\begin{array}{rcl}
[h_A(n)_{\langle n \rangle}]^{\langle n \rangle} -1 & = & h_{A/hA}(n)^{\langle n \rangle}  \\
 & = & h_{A/hA}(n+1) \\
 & = & h_A(n+1) - h_A(n) \\
 & = & h_A(n)^{\langle n \rangle} - 1 - h_A(n)
\end{array}
\end{equation}

\noindent There may or may not exist an integer $m$ such that $m^{\langle n \rangle} = 
[h_A(n)_{\langle n \rangle}]^{\langle n \rangle} -1$.  If it does not, then the first line of (\ref{eqns}) must be an equality.  If 
it does exist then the $n$-binomial expansion of $m$ begins with $\binom{k_n -1}{n}$, i.e. $MG(h_{A/hA}(n),n) = k_n -1 -n \geq 1$.  The 
second equality in (\ref{1st line}) gives that $A/hA$ has maximal growth from degree $n$ to degree $n+1$, so in fact $\left [ \frac{(J,h)}{(h)} \right ]_n$ has base locus of dimension $k_n-n-1 \geq 1$, and $[J]_n$ has base locus of dimension $k_n-n = MG(h_A(n),n)$.

Now suppose that the second line of (\ref{eqns}) is a strict inequality and the rest are equalities, i.e. we have 
\begin{equation} \label{2nd line}
\begin{array}{rcl}
[h_A(n)_{\langle n \rangle}]^{\langle n \rangle} -1 & = & h_{A/hA}(n)^{\langle n \rangle} -1  \\
 & = & h_{A/hA}(n+1)   \\
 & = & h_A(n+1) - h_A(n)  \\
 & = & h_A(n)^{\langle n \rangle} - h_A(n) -1.
\end{array}
\end{equation}

In particular, we have (a) $h_A(n)_{\langle n \rangle} = h_{A/hA}(n)$ and (b) the growth of $h_{A/hA}$ is one less than maximal from degree $n$ to degree $n+1$. From (a) and the fact that  $k_n \geq n+2$, by (\ref{hAn1}) we obtain $MG(h_{A/hA}(n),n) = (k_n-1)-n \geq 1$. 
If $(k_n -1)-n = 1$, Theorem \ref{n+1 choose n} shows that the base locus of $\left [ \frac{(J,h)}{(h)} \right ]_n$ is a scheme of dimension either 0 or 1, so $[J]_n$ defines a scheme of dimension either 1 or 2, i.e. either $MG(h_A(n),n)-1$ or $MG(h_A(n),n)$.  If instead $(k_n-1)-n \geq 2$ then by induction 
$\left [ \frac{(J,h)}{(h)} \right ]_n$ defines a scheme of dimension either $(k_n-1)-n$ or $(k_n-1)-n-1$.  Then $[J]_n$ defines a scheme of dimension either $k_n-n$ or $k_n-n-1$, i.e. either $MG(h_A(n),n)$ or $MG(h_A(n),n)-1$.

Finally, suppose that the third line of (\ref{eqns}) is a strict inequality  and the rest are equalities, i.e. we have 
\begin{equation} \label{3rd line}
\begin{array}{rcl}
[h_A(n)_{\langle n \rangle}]^{\langle n \rangle} -1 & = & h_{A/hA}(n)^{\langle n \rangle} -1  \\
 & = & h_{A/hA}(n+1) -1  \\
 & = & h_A(n+1) - h_A(n)  \\
 & = & h_A(n)^{\langle n \rangle} - h_A(n) -1.
\end{array}
\end{equation}
We again have $h_A(n)_{\langle n \rangle} = h_{A/hA}(n)$ but this time the second equation of (\ref{3rd line}) shows that the growth of $h_{A/hA}$ is maximal from degree $n$ to degree $n+1$. We again have $MG(h_{A/hA}(n),n) = k_n-n-1$.  This time maximal growth implies the base locus of $\left [ \frac{(J,h)}{(h)} \right ]$ has dimension $k_n-n-1$, so the base locus of $[J]_n$ has dimension $k_n-n = MG(h_A(n),n)$.
\end{proof}

\begin{cor} \label{sect 3 pts}

Let $Z \subset \mathbb P^r$ be a finite set of points, let $L$ be a general linear form, and let $J = \frac{(I_Z,L)}{(L)} \subset S = R/(L)$, where $R = k[x_0,\dots,x_r]$.  Let $A = S/J$.  For some integer $n$, assume that $A$ satisfies the condition of one of the two previous theorems. Then the base locus of $[I_Z]_n$ is a scheme of dimension either $MG(h_A(n),n)$ or $MG(h_A(n),n)+1$.

\end{cor}

\begin{rem}
In the setting of Corollary \ref{sect 3 pts}, it is clear that we expect many of the points to lie on the base locus.  However, we do not have a good estimate for the number of such points.  Nevertheless, we carefully analyze this question in the next   section for the one remaining type of almost maximal growth (see Proposition \ref{first estimate}).  Furthermore, in section \ref{k k-1 BPF}, even when there is no higher dimensional base locus we show the surprising result that there is a ``hidden" linear variety containing many of the points. 
\end{rem}


\section{Growth type $(\dots, k, k-1)$: when there is a non-empty base locus} \label{k k-1 BL}

The first case not covered in the last section is the following.  Let  $2 \leq k \leq n$.  Then an ideal $J \subset S = k[x_1,\dots,x_r]$ with $h_{S/J}(n) = k$ must satisfy $h_{S/J}(n+1) \leq k$ by Macaulay's theorem.  Thus ``almost maximal growth" means $h_{S/J}(n+1) = k-1 \geq 1$.  In this case we will see that $[J]_n$ may or may not have a base locus, and that {the setting of arbitrary homogeneous ideals becomes almost trivial, but the setting
of artinian reductions of ideals of sets of points is very interesting.}

Thus we now let $Z \subset \mathbb P^r$ be a finite set of reduced points.  Let $L$ be a general linear form in $R = k[x_0,\dots,x_r]$, and let $J := \frac{I_Z,L}{\langle L \rangle)} \subset S := R/\langle L \rangle \cong k[x_1,\dots,x_r]$.  
Throughout this section we make the following assumption.

\begin{equation}\label{section assumption}
\begin{array}{l}
\hbox{{\it $Z$ is a reduced finite set of points whose $h$-vector   has values $k$ and $k-1$ in}} \\
\hbox{{\it  degrees $n$ and $n+1$ respectively, where $n \geq k \geq 2$, and is zero thereafter.}} \\
\end{array}
\end{equation}
 
\smallskip

The assumption that the $h$-vector is zero starting in degree $n+2$ is without loss of generality.  Indeed, we can always remove points one at a time so that the Hilbert function is continually truncated at the top {(\cite{OR}, lemma 5 and proposition 7.iii)}.  This means that the $h$-vector is unchanged in all except the last non-zero value, where it drops by one.  We repeat this until we have the desired 0 in degree $n+2$.

We recall the relevant result from \cite{BGM} for maximal growth. (We have suitably revised the statement to fit in with our situation.)

\begin{thm}[\cite{BGM} Theorem 3.6] \label{BGM k k thm}
Let $Z \subset \mathbb P^r$ be a reduced set of points.  Assume that for some integer $n$, $\Delta h_Z(n) = \Delta h_Z(n+1) = k$, where $n \geq k \geq 1$.  Then

\begin{itemize}

\item[(a)] $\langle [I_Z]_{\leq n} \rangle$ is the saturated ideal of a curve, $V$, of degree $k$ (not necessarily unmixed). Furthermore $V$ is reduced. 

\end{itemize}

\noindent Let $C$ be the unmixed one-dimensional part of $V$.  Let $Z_1$ be the subset of $Z$ lying on $C$ and let $Z_2$ be the ``residual" set.

\begin{itemize}

\item[(b)] $\langle [I_{Z_1}]_{\leq n} \rangle = I_C$.

\item[(c)] 
\[
\Delta h_{Z_1}(t) = 
\left \{
\begin{array}{ll}
\Delta h_C(t), & \hbox{for $t \leq n+1$}; \\
\Delta h_Z(t), & \hbox{for $t \geq n$}.
\end{array}
\right.
\]
In particular, $\Delta h_{Z_1}(t) = k$ for all $k \leq t \leq n+1$.

\end{itemize}

\end{thm}

The starting point of the proof of the above result was that the maximal growth condition forces the base locus of the degree $n$ (and degree $n+1$) component of $J$ to be a zero-dimensional scheme of degree $k$ (eventually proven to be reduced), which then lifts to a curve of degree $k$ containing many of the points of $Z$.  

Unfortunately, the main results of \cite{BGM} used {Gotzmann's theorem, which holds only in the case of maximal growth}.  As a result, there are more possibilities in the setting of (\ref{section assumption}).  Thus, our analysis will be based on a consideration of the possibilities for the base locus of $[J]_n$ and $[J]_{n+1}$.  

If $Z \subset \mathbb P^2$ then the analysis was already known to Davis \cite{davis}.  Thus we will assume 
\begin{equation} \label{hvector}
\begin{array}{l}
\hbox{{\it The $h$-vector of $R/I_Z$ has the form $(1,h_1, \dots, k,k-1,0)$, where the $k$ occurs in}} \\
\hbox{{\it degree $n \geq k  \geq 2$, and $3 \leq h_1 \leq r$.}}
\end{array}
\end{equation}

\begin{lem}  \label{elim hi dim}
Under the assumptions of (\ref{hvector}), let $W$ be the base locus of $[J]_n$.  Then either $\dim W = 0$ with $\deg W \leq k$ or $W = \emptyset$.  In the latter case, the linear system $|[J]_n|$ is basepoint free.
\end{lem}

\begin{proof}
If  $W$ were one-dimensional or more (as a subscheme of the $\mathbb P^{r-1}$ defined by $L$) then the elements of $[J]_n$ all lie in $[I_W]_n$.  This means that the Hilbert function of $S/J$ is greater than or equal to that of $S/I_W$ in degree $n$, which is not possible since $h_{S/I_W}(n) \geq n+1$. This also shows $\deg W \leq k$.  Hence if the base locus of $[J]_n$ is not empty, it is zero-dimensional.  The last assertion is obvious.
\end{proof}

We will see that both possibilities allowed by Lemma \ref{elim hi dim} actually can occur.  In this section we deal with the first of the two, i.e. the case $\dim W = 0$.  We first show that the base locus can have any degree $\leq k$.

\begin{prop} \label{aaaaa}
Fix any integer $d$ with $1 \leq d \leq k \leq n$.  Then there exists a set of points $Z \subset \mathbb P^r$ satisfying (\ref{hvector}),
for which $[J]_n$ has base locus consisting  of precisely $d$ points.  That is, $[I_Z]_n$ has a base locus whose 1-dimensional component is a reduced curve of degree $d$. 
\end{prop}

\begin{proof}
We proceed in three steps.  First, fix a plane $\Gamma \subset \mathbb P^r$ and choose a smooth plane curve $C$ of degree $d$ on $\Gamma$.  Choose a general set, $Z_1$, of 
\[
1+\dots + d + d + \dots + d + (d-1) = \binom{d}{2} + (n-d+2)d + (d-1)
\]
points on $C$.  The $h$-vector of $Z_1$ is $(1,\dots,d,d,\dots,d,d-1)$ where the last $d$ occurs in degree $n$.  (The first $d$ entries increase monotonically by 1, but if $d=1$ then there is no increase.)

For the second step, if $d=k$ then we do nothing and proceed to the third step.  If, instead, we have $d<k$, then we want to add points on $\Gamma$ so that $\Delta h_{R/I_Z}(n) = k$, $\Delta h_{R/I_Z}(n+1) = k-1$, and the base locus in degree $n$ and $n+1$ remains a curve of degree $d$.  We first add one more general point to $Z_1$ on $C$, making the value of the $h$-vector in degree $n+1$ now $d$.  Let $Z_2'$ be a general complete intersection in $\Gamma$ of type
\[
\left (\left  \lfloor \frac{k+n-2d+1}{2} \right \rfloor , \left \lceil \frac{k+n-2d+1}{2} \right \rceil \right ).
\]
Observe that the $h$-vector of $Z_2'$ in degree $n-d$ is $k-d$, and the value in degree $n-d+1$ is $k-d-1$.  Since $Z_2'$ is in uniform position, we can remove any $\binom{k-d-1}{2}$ points from $Z_2'$ to obtain a set of points $Z_2$ whose $h$-vector agrees with that of $Z_2'$ up to degree $n-d+1$ and is zero thereafter.  The base locus of the ideal of $Z_2$ in degree $n-d$ is zero-dimensional.  Let $Z = Z_1 \cup Z_2$.  We claim that for $t \leq n+1$ we have 
\[
\Delta h_{R/I_Z}(t) = \Delta h_{R/I_{Z_2}}(t-d) + h_{R/I_{C}}(t).
\]
It is enough to think of these algebras as quotients of $k[x,y,z]$.  Then this follows from the work of Davis, but can also be seen directly with a simple computation using Bezout's theorem, since for $t \leq n+1$ any plane curve of degree $t$ containing $Z_1$ must contain $C$ as a component, so $\dim [I_Z]_t = \dim [I_{Z_2}]_{t-d}$.  With this formula we obtain that the $h$-vector of $Z$ has value $k$ in degree $n$ and $k-1$ in degree $n+1$ as desired.  Also, the base locus of $I_Z$ in degrees $n$ and $n+1$ has only $C$ as 1-dimensional base locus.

For the third step, observe that $Z$ as constructed so far has $h_1 = 2$.  We now merely add up to $r-2$ general points to obtain $3 \leq h_1 \leq r$, and we are finished.
\end{proof}

We would like to point out, though, that in Proposition \ref{aaaaa} there are at least 
\[
1+2+3+\dots+k + k + \dots +k + (k-1) = \binom{k}{2} + k(n-k+3) -1
\]
points on a plane, which also contains the 1-dimensional component.  It is possible to construct examples where the curve (and hence many of the points of $Z$) does not lie on a plane.  We give an example.

\begin{ex}
Consider the case of finite sets of points in $\mathbb P^3$ and suppose $n=8$ and $k=7$.  We show that for any $1 \leq d \leq 7$ there is a set of points $Z \subset \mathbb P^3$ such that 
\begin{itemize}

\item $\Delta h_{R/I_Z} (8) = 7$ and $\Delta h_{R/I_Z}(9) = 6$;

\smallskip

\item the base locus of $[I_Z]_8$ is a non-degenerate, smooth curve, $C$, of degree $d$.

\smallskip

\end{itemize}

\noindent Furthermore, we show that there is a set of points whose $h$-vector ends $(\dots,7,6)$ and whose base locus for the component in degree $8$ is zero-dimensional (in fact it consists only of points of $Z$).  This is in stark contrast to the $h$-vector $(\dots,7,7)$, where the base locus must be a curve of degree $7$.

For fixed $d$, $1 \leq d \leq 7$, let $C$ be a smooth, irreducible arithmetically Cohen-Macaulay curve of degree $d$ whose $h$-vector is as short as possible, and let $Z_1$ be a general set of points on $C$ with $h$-vector as follows:
\[
\begin{array}{c|lcccccccccc}
d & \hbox{$h$-vector of $Z_1$} \\ \hline
1 & (1,1,1,1,1,1,1,1,1) \\
2 & (1,2,2,2,2,2,2,2,2) \\
3 & (1,3,3,3,3,3,3,3,3) \\
4 & (1,3,4,4,4,4,4,4,4) \\
\end{array}
\ \ \ \ 
\begin{array}{c|lcccccccccc}
d & \hbox{$h$-vector of $Z_1$} \\ \hline
5 & (1,3,5,5,5,5,5,5,5) \\
6 & (1,3,6,6,6,6,6,6,6,6) \\
7 & (1,3,6,7,7,7,7,7,7,7) \\
& \\
\end{array}
\]
(Note that the length of the $h$-vector in the last two cases is one more than the lengths in the previous cases.)

In the case $d=7$ we simply remove any point $P$ from $Z_1$ and we are finished: $Z = Z_1 - \{P\}$ and the $h$-vector is $(1,3,6,7,7,7,7,7,7,6)$.

In the case $d=6$ we add 81 general points of $\mathbb P^3$ to $Z_1$, arriving at the $h$-vector \linebreak
$(1,3,6,10,15,21,28,36,7,6)$.

Now assume $1 \leq d \leq 5$.
Let $F_1, F_2$ be a regular sequence in $I_C$ of type $(7,9-d)$, hence also a regular sequence of type $(7, 9-d)$ in $I_{Z_1}$.  Let $L$ be a general linear form.  The ideal $\langle L, F_1, F_2 \rangle$ is the saturated ideal of a complete intersection, $Z_2$, of type $(7, 9-d)$ in the plane defined by $L$, hence of $63-7d$ reduced points.  Thanks to [\cite{KMMNP} Lemma 4.8 and Remark 4.10] or [\cite{book} Proposition 5.4.5], the ideal $L \cdot I_{Z_1} + \langle F_1, F_2 \rangle$ is the saturated ideal of the union, $Z$, of $Z_1$ and $Z_2$.  Its $h$-vector is
\[
\begin{array}{c|cccccccccccc}
\hbox{deg} & 0 & 1 & 2 & 3 & 4 & 5 & 6 & 7 & 8 & 9 \\ \hline
Z_1 & & 1 & * & * & d & d & d & d & d & d  \\
Z_2 & 1 & 2 & 3 & 4 & * & * & * & *  & 7-d & 6-d & \dots   \\ \hline
& 1 & 3 & * & * & * & * & * & *  & 7 & 6 & \dots
\end{array}
\]
By \cite{GMR}, a subset of $Z_1 \cup Z_2$ can be found by removing a suitable set of points of $Z_2$, so that the $h$-vector of the resulting set of points is $(1,3, \dots ,7,6)$ (with the 7 in degree 8).  We thus have $k = 7, n = 8$ so the numerical assumptions are satisfied.

Now, in degrees 7 and 8, $Z_2$ does not have a one-dimensional component in the base locus, hence one can check from the form of the ideal given above (and confirm on \cocoa) that the one-dimensional base locus of $Z$ in degree $8$ and degree 9 has degree $d$.

To arrange for the base locus to be zero-dimensional, we replace the set of points on a curve with any set of points with regularity less than 7 (say), and let $\langle F_1,F_2 \rangle$ be a complete intersection of type $(7,9)$ and apply the same construction.
\qed
\end{ex}

We now ask what consequences are imposed by the assumption that the base locus of $[J]_n$ (or $[J]_{n+1}$) is zero-dimensional of degree $d$.  We first recall the following result.

\begin{prop}[\cite{BGM} Corollary 3.7] \label{BGM cor 3.7}
Let $Z \subset \mathbb P^r$ be a reduced finite set of points.  Assume that for some $n \geq d$, $\Delta h_{R/I_Z}(n) = d$ and that the saturated ideal $\langle [I_Z]_{\leq n}\rangle^{sat}$ defines a curve of degree $d$.  Then all the conclusions of Theorem \ref{BGM k k thm} continue to hold.
\end{prop}

\begin{cor} \label{base locus k-1}
Assume that $[J]_{n+1}$ has a base locus consisting of a zero-dimensional scheme in $\mathbb P^{r-1}$ of degree $k-1$, together with the other assumptions highlighted so far in this section.  Then all the conclusions of Theorem \ref{BGM k k thm} continue to hold.  In particular, $[I_Z]_{n+1}$ has a reduced 1-dimensional component of degree $k-1$.
\end{cor}

\begin{prop}  \label{mingen}
Suppose that $I_Z$ has a minimal generator of degree $n+1$, together with our assumption (\ref{hvector}).   Then all the conclusions of Theorem \ref{BGM k k thm} continue to hold.  In particular, $[I_Z]_n$ has a reduced 1-dimensional component of degree~$k$.
\end{prop}

\begin{proof}
If $I_Z$ has such minimal generator then so does the artinian reduction $J$.  Removing this generator from $J$ gives a new artinian ideal $J'$ that agrees with $J$ through degree $n$, and whose Hilbert function has the value $k$ degree $n$. This means that $[J']_n$, and hence also $[J]_n$, has a base locus consisting of a zero-dimensional scheme of degree $k$.  It follows that $[I_Z]_n$ has a base locus with a one-dimensional component of degree $k$.  Thus Proposition \ref{BGM cor 3.7} applies and we are done.
\end{proof}

\begin{cor} \label{count pts sect 4}
In the setting of Proposition \ref{mingen}, at least 
\[
\binom{k}{2} + k(n-k+3) -1
\]
points of $Z$ lie on a curve of degree $k$.  In the setting of Corollary \ref{base locus k-1}, at least 
\[
\binom{k-1}{2} + (k-1)(n-k+4)
\]
points lie on a curve of degree $k-1$.
\end{cor}

\begin{proof}
We will use several times in this paper the observation that if $Z_1$ is a set of points lying on a curve $C$ of some given degree $d$ then the sum of the entries of the $h$-vector (i.e. the number of points) of $Z_1$ is minimized when $C$ is a plane curve, since then Theorem \ref{BGM k k thm} (once we know that it applies) gives that the $h$-vector is of the form $(1,2,3,\dots, (d-1),d,d,\dots,d, \epsilon$ (where $\epsilon \leq d$) so this allows the $h$-vector to grow as slowly as possible.  By Theorem \ref{BGM k k thm}, since the degree of the curve is $k$ and the curve is reduced, it follows that 
\[
|Z_1 | \geq 1+2+ \dots + (k-1) + k + k + \dots + k + (k-1)
\]
where the two $(k-1)$'s occur in degrees $k-2$ and $n+1$.   The result is then an easy calculation.  The same idea gives the second result.
\end{proof}

This leads to the next reduction.

\begin{equation}\label{assume3}
\begin{array}{l}
\hbox{{\it From now on we assume that $J$, and hence $I_Z$, has no minimal generator}} \\
\hbox{{of degree $n+1$.}}
\end{array}
\end{equation}

\medskip

We believe that something like Proposition \ref{mingen} should continue to hold.

\begin{conj} \label{base locus d conj}
Assume that $[J]_n$ has a base locus consisting of a zero-dimensional scheme in $\mathbb P^{r-1}$ of degree $d$, together with the other assumptions highlighted so far in this section.   As noted above, we have $d \leq k$.  Then

\begin{itemize}

\item[(a)] $\langle [I_Z]_{\leq n+1} \rangle$ is the saturated ideal of a {reduced}
curve, $V$, of degree $d$ (not necessarily unmixed).  

\end{itemize}

\noindent Let $C$ be the unmixed one-dimensional part of $V$.  Let $Z_1$ be the subset of $Z$ lying on $C$ and let $Z_2$ be the ``residual" set.

\begin{itemize}

\item[(b)] $\langle [I_{Z_1}]_{\leq n+1} \rangle = I_C$.

\item[(c)]  $\Delta h_{Z_1}(t) = 
\Delta h_C(t), \  \hbox{for $t \leq n+1$}.  $
In particular, $\Delta h_{Z_1}(t) = d$ for all $d \leq t \leq n+1$.

\end{itemize}

\end{conj}

\begin{quest} \label{pts on plane question}
{\it We will see in section \ref{k k-1 BPF} that in the situation where $[J]_n$ is basepoint free, many points of $Z$ have to lie on a plane.  If the above conjecture is true, does it also follow that many points of $Z_2$ have to lie on a plane?  We do not have much information about the Hilbert function of $Z_2$, even conjecturally.}
\end{quest}

\begin{lem}\label{Lemma1} Every $0$-dimensional subscheme $Z$
of length $d$ in $\pp^r$ satisfies \linebreak $H^1\ii_Z(d-1)=0$.
\end{lem}
\begin{proof} Consider the $h$-vector, $h$, of $Z$. Since the $h$-vector is the Hilbert function of $R/(I_Z,L)$ for a general linear form $L$, it follows that if $h(i) = 0$ for some $i$ then $h(i+j)=0$ for all $j>0$.
But then $h$ is, at worst,
of the form $1\ 1\ 1\ ... 1\ 0\ ...$ so that $h(d)=0$, and the conclusion follows.
\end{proof}

By regularity, it follows that every $0$-dimensional subscheme $Z$
of length $d$ has a homogeneous ideal generated in degree $\leq d$.

\begin{lem}\label{Lemma2} Given an unmixed reduced curve $C$ of degree $d$ 
(even reducible) in $\pp^r$, then $H^1\ii_C(d-1)=0$.
\end{lem}
\begin{proof} Well known.  See for instance Remark \ref{GLP remark}.
\end{proof}

Here is the main new fact.

\begin{prop} \label{first estimate}
Assume that the  $h$-vector of $R/I_Z$ 
has the form $(1,h_1, \dots, k,k-1,0)$, where the $k$ occurs in
degree $n \geq k \geq 2$. Assume that $J$, and hence $I_Z$, has no minimal generator
of degree $n+1$. Assume that $[J]_n$ has a base locus consisting of a zero-dimensional scheme
in $\pp^{r-1}$ of degree $d'$. Call $C'$ the union of all positive
dimensional parts of the saturation of  $\langle [I_Z]_{\leq n} \rangle$, so $\deg C' = d'$.
Let $C=(C')_{red}$ and assume that $\deg C = d$.

Then $C$ has dimension $1$ and degree $d\leq k-1$.  Furthermore, 

\begin{itemize}

\item if $k-1-d \leq d+1$ then $C$ contains at least $\displaystyle 2 \cdot \binom{d}{2} + 5d$ points of $Z$;

\item If $d+1 \leq k-1-d$ then $C$ contains at least $d^2 + d(n-k+4)$ points of $Z$

\end{itemize}

\end{prop}
\begin{proof}
The fact that $C'$ is a (possibly non reduced) curve of degree $d'$ contained in the base locus of $[I_Z]_n$ 
follows at once from the fact that $L$ is a general hyperplane.

Call $Z_1$ the subset of $Z$ lying on $C'$. Notice that indeed $Z_1$ lies on
$C$. Define $Z-C=Z_2$, where $Z-C$ is defined by the homogeneous ideal $I_Z : I_C$.  Since $Z$ is reduced, we have $Z = Z_1 \cup Z_2$.

Since $I_Z$ has no minimal generators in degree $n+1$, we have for any $i \leq n+1$ that 
\[
[I_Z]_{i} \subseteq [I_{C \cup Z_2}]_{i} = [I_{C \cup Z}]_{i} \subseteq [I_Z]_{i},
\]
so all of these vector spaces are equal.  In particular, $I_Z$ coincides up to degree $n+1$ with $I_{C\cup Z} = I_{C \cup Z_2}$.

Let $L_1 \cong \mathbb P^{r-1}$ be a general hyperplane and call $W$ the intersection
of $L_1$ and $C$. Note first that the $h$-vector of $Z'=Z\cup W$
coincides with the $h$-vector of $Z$ for $i=0,\dots,n+1$. Indeed
$Z'$ sits in the base  locus of $\langle [I_Z]_{\leq n+1} \rangle$.
In particular the $h$-vector of $Z'$ is $k-1$ in degree $n+1$.

We claim that the $h$-vector of $Z'$ is $0$ in degree $n+3$,
i.e. that $Z'$ is separated by forms of degree $n+2$.
To see this, note that the sequence
$$
0 \to H^0 \ii_Z(n+1)\stackrel{\times L_1}{\longrightarrow} H^0 \ii_Z(n+2) \to  H^0\oo_{L_1}(n+2)
\to 0
$$
is exact (identifying $L_1$ with $\mathbb P^{r-1}$), because the $h$-vector of $Z$ is $0$ in degree $n+2$,
hence $H^1\ii_Z(n+1)=0$. With this in mind, we
can prove that if $Y$ is any subscheme of
$Z'$ whose length is the length of $Z'$ minus $1$, then  
there exists a form $f$ of degree $n+2$ passing through $Y$ and not
through $Z'$.  

\medskip

\noindent $\blacktriangleright$ If $Z'-Y$ is supported in $Z$, then
one takes a form $f'$ of degree $n+1$ passing through $Z \cap Y$
(it exists because $Z$ is separated in degree $n+1$) and takes
$f=f'L_1$.

\medskip

\noindent $\blacktriangleright$ If $Z'-Y$ is supported in $L_1$, then one takes 
one form of degree $n+2$ in $L_1=\pp^{r-1}$ which passes through
$Y\cap W$ and not through $W$ (it exists, by Lemma \ref{Lemma1}) 
and lifts it to a form $f$ containing $Z$, via the surjection
\[
H^0 \ii_Z(n+2) \to  H^0\oo_{L_1}(n+2)\to 0.
\]

\medskip

\noindent It follows that the $h$-vector of $Z'$ is 
$k, k-1 , d , 0$ in degrees $n, n+1, n+2, n+3$
respectively.
This proves in particular that $d\leq k-1$. 

If $d=k-1$, we apply Theorem  \ref{BGM k k thm} 
to $Z'$ and prove that $C$ has degree $d=k-1$. 
Moreover $C$ contains a subscheme of $Z'$ of length at least 
$(k-1)(n+2-(k-1))$. Since the length of $Z' - Z$ is $d=k-1$, the
conclusion follows. 

Assume then that $d<k-1$.  By construction 
the base locus of $[I_{Z'}]_{n+1}$ contains $C$.  We do not know if the base locus of $[I_{Z'}]_{n+2}$ contains $C$, so we have to consider both possibilities. 

\medskip

\noindent {\bf Case 1:} {\it The base locus of $[I_{Z'}]_{n+2}$ does not contain $C$.}    Call $q_1$ the dimension of
the vector space 
\[
[I_{Z'}]_{n+2}/ ([I_{Z'}]_{n+2})\cap [I_C]_{n+2}).
\]
We would like to add to $Z'$ a set of $q_1$ ``general'' points of $C$.  From a geometric point of view, the significance of the number $q_1$  is the following.  

\begin{quotation}

{\it This number $q_{1}$ is the number of points on $C$ that we need to add to \linebreak $Z \cup (L_{1}\cap C)$   (where $L_{1}$ is a general plane) in order to get a set $Z'_{1}$ such that all the elements of $H^0\mathcal I_{Z'_{1}}(n+2)$ contain $C$.}

\end{quotation}

\noindent We note here that the numbers $q_i$ introduced below have an analogous geometric meaning:  $q_i$ is the number of points on $C$ that we need to add to $Z'_{i-1} \cup (L_i \cap C)$ in order to get a set $Z'_i$ such that all elements of $H^0 \mathcal I_{Z'_i}(n+1+i)$ contain $C$.

The first problem is what this means if $C$ is not irreducible.  This will be resolved by the following procedure.  Let $\{ F_1,\dots,F_m \}$ be a basis for $([I_{Z'}]_{n+2} \cap [I_C]_{n+2})$.  We have that $\dim [I_{Z'}]_{n+2} = m + q_1$, so we can complete this to a basis for $[I_{Z'}]_{n+2}$ as follows.  It is possible that $[I_{Z'}]_{n+2}$ has a one-dimensional base locus, $C_1$, whose support is a union of some (but not all) components of $C$.  Choose a component of $C$ that is not a component of $C_1$, and let $P_1$ be a general point of this component and let $A_1 = Z' \cup \{P_1\}$.  Clearly $\dim [I_{A_1}]_{n+2} = m +q_1 -1$.  Applying the same procedure but with $A_1$ instead of $Z'$, we obtain a point $P_2$ on some component of $C$ and form $A_2 = Z' \cup \{P_1,P_2\}$ with $\dim [I_{A_2}]_{n+2} = m+q_1-2$.  Continuing in this way, we obtain a set of points $\{P_1,\dots,P_{q_1}\}$ so that if $Z_1' = Z' \cup \{P_1,\dots,P_{q_1}\}$ then $\dim [I_{Z_1'}]_{n+2} = \dim [I_{Z'}]_{n+2} - q_1$.  Since $\dim[I_{Z'_1}]_{n+1} = \dim [I_{Z'}]_{n+1}$ (the new points all lie in the base locus), we see that the length of the resulting scheme $Z'_1$ is length$(Z)+d+q_1$,  its $h$-vector is equal to the $h$-vector of $Z$ up to degree $n+1$, and the $h$-vector has value $d+q_1$ in degree $n+2$.

It follows that $Z'_1$ is separated in degree $n+2$ and it has
length $d+q_1+$length$(Z)$. Moreover $d+q_1\leq k-1$.
If $d+q_1=k-1$, the $h$-vector of $Z'_1$ ends with $k-1, k-1$
in degrees $n+1, n+2$. As above, we can apply
Theorem  \ref{BGM k k thm} and find that 
$C$ has degree $d=k-1$, and contains a subset of $Z'_1$
of length at least $(k-1)(n+2-(k-1))$. 
Since the length of $Z'_1 - Z$ is $d+q_1=k-1$, the
conclusion follows.  
{ As above, we can apply Theorem 4.1.  We obtain that $I_{Z_1'}$ has no minimal 
generator in degree $n+2$, and that the base locus of the components in degrees
$n+1$ and $n+2$ is a curve of degree $k-1$.  Since we also assumed that there
is no minimal generator in degree $n+1$, we get that this curve is $C$, and that
hence $d = k-1$ and $q_1 = 0$ (we get a little bit more: we get that the curve is $C'$, i.e. $C'$ is reduced).  The estimate for the number of points of $Z$ that 
lie on $C$ will be carried out at the end of the proof}.

Assume $d+q_1<k-1$. Fix another general hyperplane $L_2$
and repeat the procedure, i.e. add to
$Z'_1$ the set $L_2\cap C$ and, if the dimension $q_2$ of
$$[I_{Z'_1\cup(L_2\cap C)}]_{n+3}/ ([I_{Z'_1\cup(L_2\cap C)}]_{n+3}
\cap [I_C]_{n+3})$$
is non-zero, also add a set of $q_2$ ``general'' points on $C$ as above.
Call $Z'_2$ the resulting set. As above, we have a surjection
$$H^0 \ii_{Z'_1}(n+3) \to  H^0\oo_{L_2}(n+3)\to 0$$
which guarantees that the $h$-vector of $Z'_2$ ends with
$k-1, d+q_1, d+q_2, 0$ in degrees $n+1, n+2, n+3 , n+4$
respectively.

We continue in the same way obtaining, after $i$ steps,
a $0$-dimensional scheme $Z'_i$ whose $h$-vector ends
with $k-1, d+q_1, \dots, d+q_i, 0$ in degrees 
$n+1, n+2, \dots, n+i+1 , n+i+2$
respectively.

We now claim that the $q_i$ are strictly decreasing until they reach the value 0.  If this were not the case, then we would have two consecutive degrees in which the first difference of the Hilbert function of our set of points $Z_i'$ would have the same value $d+q_i$.  This forces the base locus of $[I_{Z_i'}]_{n+1+i}$ to contain a curve of degree $d+q_i > d$.  But the artinian reduction of $R/I_{Z_i'}$ agrees with that of $R/I_Z$ in degrees $n$ and $n+1$, and this base locus has only degree $d$.  The base locus cannot grow as the degree increases, so this is impossible.

Thus the process stops after $m$ steps, where $m$ is the first value of $i$ such that $q_i= 0$, since then $q_i = 0$ for all $i > m$ as well, and so we have maximal growth. 

\medskip

\noindent \underline{Claim}: {\it $m \leq \min \{ k-1-d , d+1 \}$.}

\medskip

The fact that the number of steps is at most $k-1-d$ is clear: since $d + q_1 < k-1$, we get $q_1 \leq k-d-2$ so there are at most $k-d-2$ non-zero $q_i$, plus $q_m = 0$.  We will show that in fact $q_{d+1} = q_{d+2} = 0$.
To see this, note first that $C \cup Z_i' = C \cup Z_2$ for any $i$, since the added points all lie on $C$, and recall that  $Z_2$ is separated in degree $n+1$.
It follows that the sequence
\[
0\to H^0\ii_{C\cup Z_2}(d+n+1)\to H^0\ii_C(d+n+1)\to H^0\oo_{Z_2}\to 0
\]
is exact.  Indeed, the only issue is the surjectivity of the last map.  But if $P$ is a point of $Z_2$ then there is a form $F$ of degree $n+1$ vanishing at the remaining points of $Z_2$ but not at $P$, and since $I_C$ is generated in degree $\leq d$ there is a form $G$ of degree $d$ vanishing on $C$ but not at $P$.  The product $FG$ demonstrates the desired surjectivity.  

As a consequence, we have 
\[
H^1 \mathcal I_{C \cup Z_2}(d+n+1) \subseteq H^1 \mathcal I_C(d+n+1) = 0.
\]
Hence the sequence
\[
0 \rightarrow H^0 \mathcal I_{C\cup Z_2}(d+n+1) \to H^0 \mathcal I_{Z_2}(d+n+1) \to H^0\mathcal O_C(d+n+1)\to 0
\]
is exact.
It also means that for any $i$ the map 
\[
H^0 \mathcal I_{Z'_{i}}(d+n+1) \to H^0\mathcal I_{D|C}(d+n+1)
\]
surjects, where $D$ is the set on $C$ given by $C\cap Z'_{i} $.  Of course $H^0\mathcal I_{D|C}(d+n+1)$ may be zero, and certainly is zero for large $i$.

Now we prove that $q_{d+1}$ is $0$.  As noted above, this number $q_{d+1}$ is the number of points on $C$ that we need to add to $Z'_{d} \cup (L_{d+1}\cap C)$   (where $L_{d+1}$ is a general plane) in order to get a set $Z'_{d+1}$ such that  the elements of $H^0\mathcal I_{Z'_{d+1}}(n+d+2)$  all contain $C$.
So, if we can prove that $Z'_{d} \cup (L_{d+1}\cap C)$ has this property (i.e. surfaces of degree $n+d+2$ through $Z'_{d} \cup (L_{d+1}\cap C)$  all contain $C$), then we must have $q_{d+1}=0$.

To see this, suppose to the contrary that there exists a surface $F$ of degree $n+d+2$ passing through $Z'_{d} \cup (L_{d+1}\cap C)$ and not containing $C$. Then $F$ cuts a divisor $E$ in $\mathcal O_C(n+d+2)$ which contains $D = C\cap Z'_{d} $ and $L_{d+1}\cap C$.  Then $E' =  E - (L_{d+1}\cap C)$ is a non-trivial divisor in $O_C(n+d+1)$ containing $D$. Thus, by the surjection above, $E'$ can be lifted to a surface $F'$ of degree $d+n+1$ which contains $D$ and $Z_2$, i.e. $F'$ contains $Z'_{d}$ and does not contain C.  This is impossible by the construction of $Z'_{d}$.

Thus $q_{d+1} = 0$.  In the same way we have $q_{d+2} = 0$, i.e. $Z'_{d+2}$ has an $h$-vector which ends in degree $d+n+3$ and has value $d$ in degrees $d+n+2$ and $d+n+3$.

We now give the estimate for the number of points of $Z$ lying on $C$ in Case 1.  We know that after $m$ steps the one-dimensional part of the base locus is already $C$, without adding points, and that thus in degree $n+1+m$, the decomposition of $Z_m'$ according to Theorem \ref{BGM k k thm} holds.

The total number of points of $Z_m'$ that lie on $C$ is minimized in the case that $C$ is a plane curve, so that the $h$-vector grows as slowly as possible to reach the value $d$. Thus the number of points of $Z_m'$ that lie on $C$ is at least
\[
(1 + 2 + \dots + (d-1) + d + d + \dots + d 
\]
where the final $d$ occurs in degree $n+1+m$.  Thus $Z_m'$ contains at least 
\[
\binom{d}{2} + d[ (n+1+m) - (d-2)] = \binom{d}{2} + d(n+m-d+3).
\]
To obtain $Z_m'$ we added to $Z$ a total of $md + (q_1 + q_2 + \dots + q_m)$ points (where $q_m = 0$).  Bearing in mind that the $q_i$ are strictly decreasing as long as they are positive, and that  $q_1 \leq k-2-d$, we have the following lower bound for the cardinality of the original set, $Z_1$, of points lying on $C$.
\[
|Z_1| \geq \binom{d}{2} + d(n+m-d+3) - \left [ md + (q_1 + \dots + q_m) \right ] = \binom{d}{2} + d(n-d+3) - [q_1 + \dots + q_m].
\]

By the above claim, we have $m \leq \min \{ k-1-d,d+1 \}$, where $m$ is the smallest index such that $q_m = 0$.  This gives us two possibilities.

Assume first that $m \leq k-1-d \leq d+1$, i.e. $k \leq 2d+2$.  Notice that this implies $d- (k-2-d) \geq 0$, and recall also that $k \leq n$ and that $q_1 \leq k-2-d$.  Then 
\[
\begin{array}{rcl}
|Z_1| & \geq & \displaystyle \binom{d}{2} + d(n-d+3) - [1 + 2 + \dots + (k-2-d)] \\ \\
& \geq &  \displaystyle \binom{d}{2} + d(k-d+3) - [1 + 2 + \dots + (k-2-d)] \\ \\
& \geq & \displaystyle \binom{d}{2} + \left [ (d-1) + (d-2) + \dots + (d - (k-2-d)) \right ] + 5d \\ \\
\end{array}
\]
\[
\begin{array}{rcl}
& \geq & \displaystyle \binom{d}{2} + \left [ (d-1) + (d-2) + \dots +1 +0 \right ] + 5d   \\ \\
& = & \displaystyle  \binom{d}{2} + \binom{d}{2} + 5d \\ \\
& = & \displaystyle 2 \cdot \binom{d}{2} + 5d.
\end{array}
\]

Now assume that $m \leq d+1 \leq k-1-d$, i.e. $k \geq 2d+2$.  We have that $q_{d+1} = 0$, so we maximize $q_1 + \dots + q_m$ by $(k-2-d) + \dots + (k-2d-1)$ since the $q_i$ are strictly decreasing and there are at most $d$ non-zero ones.
\[
\begin{array}{rcl}
|Z_1| & \geq & \displaystyle \binom{d}{2} + d(n-d+3) - \left [ (k-2-d) + \dots + (k-2d-1)  \right ] \\ \\
& = & \displaystyle \binom{d}{2} + d(n-d+3) - \frac{d}{2} [ 2k-3d+-3] \\ \\
& = & d^2 + d (n-k+4)

\end{array}
\]

\bigskip

\noindent {\bf Case 2:} {\it The base locus of $[I_{Z'}]_{n+2}$ contains $C$.}   In this case $q_1$ and $q_2$ are both zero, so the computation is as before but easier.  Clearly in this case there are more points than in Case 1, so the claimed bound still holds.
\end{proof}


\section{Growth type $(\dots, k, k-1)$: the basepoint free case} \label{k k-1 BPF}

We continue to assume that $\dim [S/J]_n = k \geq 2$ and $\dim [S/J]_{n+1} = k-1$. The other possibility allowed by Lemma \ref{elim hi dim} is that the linear systems $|[J]_{n}|$ and $|[J]_{n+1}|$ are basepoint free,  i.e. that $\langle [J]_{\leq n} \rangle$ is artinian (hence also $\langle [J]_{\leq n+1} \rangle$ is artinian):

\begin{equation}\label{assume4}
\begin{array}{l}
\hbox{{\it From now on we assume that the degree $n$ component of $J$ has no base locus.}} \\
\hbox{{\it Hence the same is true also in degree $n+1$.}}
\end{array}
\end{equation}

\medskip

Our goal in this section is to prove the following theorem.

\begin{thm} \label{no base locus}
Assume that $[J]_n$ and $[J]_{n+1}$ are base point free, together with the other assumptions highlighted so far in this section.  Then there is a distinguished plane $\Lambda$ such that if $Z_1$ is the subset of $Z$ lying on $\Lambda$ and $Z_2$ is the subset of $Z$ lying off $\Lambda$ then 
\[
\Delta h_{R/I_{Z_1}}(t) = \Delta h_{R/I_Z}(t) \hbox{ \ \ \ for } t \geq n \hbox{ \ \ \ \ \ and \ \ \ \ \ } \Delta h_{R/I_{Z_2}}(t) = 0 \hbox{ \ \ \ for } t \geq n-1.
\]

\end{thm}

\noindent This will take some preparation, and we devote the rest of this section to its proof.

We know that $S/J$ is artinian, with Hilbert function having values $k$ and $k-1$ in degrees $n$ and $n+1$ respectively, and then zero.  Clearly the dimension of the socle of $S/J$ in degree $n+1$ is $k-1$.  Suppose that $S/J$ has a non-zero socle element $f$ in degree $n$.  Then $S/(J,f)$ has Hilbert function with value $k-1$ in both degrees $n$ and $n+1$.  Thus the components of $(J,f)$ in degrees $n$ and $n+1$ have a non-empty base locus.  But in degree $n+1$, $J$ and $(J,f)$ coincide.  Thus we have a contradiction to (\ref{assume4}).

\begin{equation}\label{assume5}
\begin{array}{l}
\hbox{{\it From now on we assume that  $S/J$ has no non-zero socle in degree $n$.}} \\
\end{array}
\end{equation}

\medskip

It follows that the canonical module of $S/J$ has no minimal generators in the {\it second} degree.  Consequently, the same is true for the canonical module of $R/I_Z$:

\begin{equation}\label{assume6}
\begin{array}{l}
\hbox{{\it From now on we assume that the canonical module of $R/I_Z$ has no minimal}} \\
\hbox{{\it generator in its second degree.}}
\end{array}
\end{equation}

With the assumptions (\ref{assume3}), (\ref{assume4}), (\ref{assume5}), (\ref{assume6}) in place, we now prove the geometric consequences as described in the theorem.

  Let $\ell$ be an arbitrary (not general) linear form.  We have a multiplication 
\[
\times \ell : [S/J]_n \rightarrow [S/J]_{n+1},
\]
where these components have dimensions $k$ and $k-1$ respectively.   Choosing bases for $[S/J]_n$ and $[S/J]_{n+1}$ and finding the matrices for $\times x_1$, $\times x_2$, \dots,  $\times x_r$ with respect to these bases, we can represent the multiplication by a linear form $\ell = a_1 x_1 + a_2 x_2 + \dots + a_r x_r$ as a $(k-1) \times k$ matrix, $A$, whose entries are linear forms, $m_{i,j}$, in the dual variables $a_1,\dots,a_r$.  Let us write
\[
A = 
\left [
\begin{array}{cccccccc}
m_{1,1} & \dots & m_{1,k} \\
\vdots && \vdots \\
m_{k-1,1} & \dots & m_{k-1,k}
\end{array}
\right ]
\]
The vanishing locus of the maximal minors of this matrix corresponds to the set of  linear forms $\ell$ for which $\times \ell$ fails to have rank $k-1$.  This locus has codimension at least 2 in $\mathbb P^{r-1}$ (the dual projective space).

\begin{rem}
As an aside, suppose that $r=3$.  Notice that the expected  vanishing locus for the maximal minors of a $(k-1) \times k$ matrix of linear forms is a finite set of $\binom{k}{2}$ points in the dual projective plane, corresponding to $\binom{k}{2}$ linear forms.  If it happens that the degree $n$ component of $J$ has a base locus consisting of $k$ distinct points, no three on a line, then the linear forms that fail to give a surjectivity are exactly the $\binom{k}{2}$ lines passing through two of the points.
\end{rem}

We first note that if $\ell$ is {\it general} then $\times \ell$ is surjective.  Indeed, we have the exact sequence
\[
[S/J]_n \stackrel{\times \ell}{\longrightarrow} [S/J]_{n+1} \rightarrow [S/(J,\ell)]_{n+1} \rightarrow 0
\]
and Green's theorem (cf. Theorem \ref{gr}) together with the assumption $k \leq n$ gives that the last vector space is zero.

Now we want to describe the possible ranks for arbitrary $\ell$.  There are essentially four possibilities for the rank of $\times \ell$, a priori:

\medskip

\noindent \underline{Case 1}:  $\hbox{rk } (\times \ell) = k-1$ (i.e. is surjective) for all $\ell$.  

This is impossible.   As noted above, the vanishing locus of the maximal minors of $A$ has codimension $\leq 2$, so there is at least a finite number of points (depending also on $r$) in the dual projective space where this rank is $< k-1$.  

\bigskip

For the remaining cases, we assume that $\times \ell$ fails to have maximal rank, and we consider the exact sequence from (\ref{ses})
\begin{equation} 
0 \rightarrow [S/(J:\ell)]_{i-1} \stackrel{\times \ell}{\longrightarrow} [S/J]_i \rightarrow [S/(J,\ell)]_i \rightarrow 0
\end{equation}
for different values of $i$.  We are interested in knowing what are the consequences for $\ell$ if it fails to give a multiplication of maximal rank.  Note that $\hbox{rk }(\times \ell) = \dim [S/(J:\ell)]_{i-1}$.

\medskip

\noindent \underline{Case 2}: For $i = n+1$, $\hbox{rk } (\times \ell) = k-2$, so $\dim [(S/ (J:\ell)]_{n} = k-2$.

\medskip

We represent the relevant data from (\ref{ses}) in the following table.  Since (\ref{ses}) is a short exact sequence of graded modules, we obtain a collection of short exact sequences of vector spaces, coming from the homogeneous components of the various degrees, $i$.  Each column in the table below represents the dimensions of the corresponding vector spaces for the degree given in the top row.  By exactness, the second and third entries must add up to the first entry.

\medskip

\begin{tabular}{r|ccccccccccc}
degree $i$ & 0 & 1 & $\dots$ & $n$ & $n+1$     \\  \hline
$\dim [S/J]_i$ & 1 & $r$ & $\dots$ & $k$ & $k-1$    \\
$\dim [S/(J:\ell)]_{i-1}$ &  & 1 & \dots & $p$ & $k-2$   \\
$\dim [S/(J,\ell)]_i$ &  1 & $r-1$ & $\dots$ & $k-p$ & 1
\end{tabular}

\medskip

\noindent where $p = \dim [S/(J:\ell)]_{n-1} \leq k$. For simple Hilbert function reasons, we cannot have $0 \leq p \leq k-3$ (looking at the middle line) or $p=k$ (looking at the bottom line).  If $p = k-1$ then $S/(J,\ell)$ has maximal growth from degree $n$ to degree $n+1$, and hence the base locus of  the component of $(J,\ell)$ in degree $n$ (and $n+1$) is non-empty, meaning that the same is true of $J$, violating (\ref{assume4}).  (Notice that we do not need $\ell$ to be general in order to reach this conclusion.)  Finally, if $p = k-2$ then $S/(J:\ell)$ has maximal growth from degree $n-1$ to degree $n$, so $[J:\ell]_n$ has a base locus.  But $(J:\ell) \supset J$, so this means that $[J]_n$ has a base locus, again violating (\ref{assume4}). 

\bigskip

Having established the idea in Case 2, we combine most of the remaining cases as Case~3.

\bigskip

\noindent \underline{Case 3}: $1 \leq s = \hbox{rk } (\times \ell) \leq k-3$. 

\medskip

Then we have the following table:

\bigskip

\begin{tabular}{r|ccccccccccc}
degree $i$ & 0 & 1 & $\dots$ & $n$ & $n+1$     \\  \hline
$\dim [S/J]_i$ & 1 & $r$ & $\dots$ & $k$ & $k-1$    \\
$\dim [S/(J:\ell)]_{i-1}$ &  & 1 & \dots & $p$ & $s$   \\
$\dim [S/(J,\ell)]_i$ &  1 & $r-1$ & $\dots$ & $k-p$ & $k-1-s$
\end{tabular}

\bigskip

\noindent Now for Hilbert function reasons $p$ cannot be equal to $0,1,\dots,s-1$ (looking at the second row).  Looking at the last row, $k-p$ cannot be equal to $0,1,\dots,k-2-s$; that is, we cannot have $p \geq s+2$.  This means that we only have to consider $p=s$ and $p = s+1$.  For this we argue exactly as in Case 2.

\bigskip

\noindent \underline{Case 4}:  $\hbox{rk } (\times \ell) = 0$.  

Since the other cases have been ruled out by our assumptions, this case must hold.  This means that $\times \ell$ is the zero map on $S/J$ from degree $n$ to degree $n+1$, i.e. that $[J]_{n+1} = [(J,\ell)]_{n+1}$.  
Thus $[(I_Z,L)]_{n+1} = [(I_Z, L, \bar \ell)]_{n+1}$ for any lifting $\bar \ell$ of $\ell$ to $R$ (note that $L$ and $\bar \ell$ are independent).

We now have the table

\bigskip

\begin{tabular}{r|ccccccccccc}
degree $i$ & 0 & 1 & $\dots$ & $n$ & $n+1$     \\  \hline
$\dim [S/J]_i$ & 1 & $r$ & $\dots$ & $k$ & $k-1$    \\
$\dim [S/(J:\ell)]_{i-1}$ &  & 1 & \dots & $p$ & 0   \\
$\dim [S/(J,\ell)]_i$ &  1 & $r-1$ & $\dots$ & $k-p$ & $k-1$
\end{tabular}

\bigskip

\noindent It follows immediately that $p = 0$, i.e. that $(\times \ell)$ is also the zero map from degree $n-1$ to degree $n$, and hence that $[J]_{n} = [(J,\ell)]_{n}$.  In particular, $[J]_n$ contains $\ell \cdot [S]_{n-1}$ and $[J]_{n+1}$ contains $\ell \cdot [S]_n$.
Thus $[(I_Z,L)]_{n} = [(I_Z, L, \bar \ell)]_{n}$ for any lifting $\bar \ell$ of $\ell$ to $R$.

\medskip

We summarize these facts:

\begin{equation}\label{assume8}
\begin{array}{l}
\hbox{{\it {\rm (a)} There exists a linear form $\ell \in S = R/(L)$ such that $\times \ell : [S/J]_{n-1} \rightarrow [S/J]_n$ }} \\
\hbox{{\it and $\times \ell : [S/J]_n \rightarrow [S/J]_{n+1}$ are both the zero map.}} \\  \\

\hbox{{\it {\rm (b)} For any lifting $\bar \ell$ of $\ell$ to $R$, we have $[J]_i = [(J,\ell)]_i $, i.e. $[(I_Z,L)]_i = $ }} \\
\hbox{{\it $[(I_Z,L,\bar \ell)]_i$    for $i = n$ and $n+1$. }}

\end{array}
\end{equation}

\medskip

By duality,

\begin{equation}\label{assume9}
\begin{array}{l}
\hbox{{\it The linear form $\ell$ also annihilates the first and second components of the }} \\
\hbox{{\it canonical module of $S/J$.}}
\end{array}
\end{equation}

We now return to the matrix, $A = (m_{i,j})$, that defines the multiplication from degree $n$ to degree $n+1$.  Since $\times \ell$ is surjective for general $\ell$, $A$ drops rank in codimension 1 or codimension 2.  On the other hand, we have seen that when $\times \ell$ is not surjective, it is the zero map.  We obtain:

\begin{lem}
The ideal $\langle m_{1,1}, \dots, m_{k-1,k} \rangle$ generated by the entries of $A$ defines a codimension 2 linear variety in the dual projective space $\mathbb P^{r-1}$.  Thus up to change of variables we may assume that the $m_{i,j}$ involve only two (dual) variables.  In particular, there exist $r-2$ independent linear forms in $[S]_1$ that annihilate $[S/J]_{n-1}$ and $[S/J]_{n}$.  When $r=3$, $\ell$ is unique up to scalar multiplication.
\end{lem}

\begin{proof}
We know that the set of $\ell$ for which $\times \ell$ is the zero map is a set of codimension $\leq 2$, and it is defined by the ideal $\langle m_{1,1}, \dots, m_{k-1,k} \rangle$.  Clearly this defines a linear variety.  Suppose that this variety has codimension 1.  Then there is a linear form in the dual variables $a_1,\dots,a_r$ such that each $m_{i,j}$ is a scalar multiple of this linear form.  Changing basis if necessary, we can suppose that this linear form is $a_1$.  This means that for any linear form $\ell$ involving only $x_2,\dots,x_r$, we have $\times \ell : [S/J]_n \rightarrow [S/J]_{n+1}$ is the zero map, and the same is true replacing $n$ by $n-1$ as noted above.  This means that
\[
\langle x_2,\dots,x_r \rangle \cdot [S]_{n-1} \subset [J]_n.
\]
But this forces $\dim [S/J]_n$ and $\dim [S/J]_{n+1}$ both to be $\leq 1$, contradicting our assumption that $k > 1$.
Thus we have shown the first assertion.  
The second, third and fourth statements are then immediate. 
\end{proof}

\begin{lem} \label{images}
Let $M$ be a finitely generated graded $R$-module of depth $\geq 1$ and let $L$ be a general linear form.  Let $N = M/LM$.  Without loss of generality say that the initial degree of $M$ is 0.  Assume that there exist linear forms $\ell_1,\dots,\ell_{r-2} \in S = R/(L)$ that annihilate $N_0$ and $N_1$.  For each $i$ let $\bar \ell_i$ be any lifting of $\ell_i$ to $R$.  Then for each $i$, the image of $\times \bar \ell_i : M_0 \rightarrow M_1$ is contained inside the image of $\times L : M_0 \rightarrow M_1$, and similarly for $\times \bar \ell_i : M_1 \rightarrow M_2$.  For a general lifting $\bar \ell$, these images actually coincide.
\end{lem}

\begin{proof}

The first assertion follows from the following commutative diagram ($j = 0, 1$):
\[
\begin{array}{cccccccccc}
0 &&  0 \\
\downarrow &&  \downarrow \\
M_{j-1} & \stackrel{\times \bar \ell_i}{\longrightarrow} & M_j \\
 \phantom{{\scriptstyle \times L}} \downarrow {\scriptstyle \times L} && \phantom{{\scriptstyle \times L}} \downarrow {\scriptstyle \times L} \\
M_j & \stackrel{\times \bar \ell_i}{\longrightarrow} & M_{j+1} \\
\downarrow && \downarrow \\
N_j & \stackrel{\times  0}{\longrightarrow} &  N_{j+1} \\
\downarrow &&  \downarrow \\
0 &&  0 \\
\end{array}
\]
The second assertion follows from a diagram chase, since a general $\bar \ell$ will be a general element of the pencil spanned by $L$ and $\bar \ell$, so it will be a non-zerodivisor for $M$, since $L$ is.
\end{proof}

\begin{cor} \label{module cor}
In the setting of Lemma \ref{images}, suppose that $L$ and $L'$ are general linear forms, and let $\ell_1,\dots,\ell_{r-2}$ and $\ell'_1,\dots,\ell'_{r-2}$ be  corresponding linear forms in $R/L$ and $R/L'$ respectively that annihilate the components $N_i$ and $N'_i = [M / L' M]_i$, respectively, where $i = 0,1$.  Let $\bar \ell_i$ and $\bar \ell'_i$ be general liftings.  If $M$ has no minimal generator in degree 1, and $\dim M_1 > 2 \cdot \dim M_0$ then $L, L', \bar \ell_1,\dots,\bar \ell_{r-2}, \bar \ell'_1, \dots, \bar \ell'_{r-2}$  define a point in $\mathbb P^r$.
\end{cor}

\begin{proof}
We have seen that $\ell_1,\dots,\ell_{r-2}$ are independent in $[S]_1$, so $L, \bar \ell_1,\dots,\bar \ell_{r-2}$ already define a line in $\mathbb P^r$, as do $L', \bar \ell'_1,\dots,\bar \ell'_{r-2}$.  Since $L'$ is general with respect to $L$ and all the $\ell_i$, the linear forms define at most a point.
If the statement is not true, then together all these linear forms span $[R]_1$.  Let us compute the dimension of the image of the multiplication
\[
\phi : [R]_1 \otimes M_0 \rightarrow M_1.
\]
By Lemma \ref{images}, the image of $\times L : M_0 \rightarrow M_1$ contains the image of $\times \bar \ell_i : M_0 \rightarrow M_1$ for each $i$, and similarly the image of  $\times L' : M_0 \rightarrow M_1$ contains the image of $\times \bar \ell'_i : M_0 \rightarrow M_1$ for each $i$.  Hence if $m_1,\dots, m_s$ form a basis of $M_0$, then $Lm_1,\dots, Lm_s, L'm_1, \dots, L' m_s$ are a basis for the image of $\phi$, so the image has dimension $\leq 2 \cdot \dim M_0$.  Since $M$ has no minimal generator in degree 1, we have a contradiction.  The result follows immediately.
\end{proof}

Now let $M$ be the canonical module of $R/I_Z$, suitable shifted so that it begins in degree~0. Note that $N$ is not the canonical module of $S/J$, but it is a twist of this module.  We have $\dim M_0 = k-1$ and $\dim M_1 = 2k-1$. We know that $M$ has no minimal generator in degree 1 by (\ref{assume6}).  Thanks to (\ref{assume9}), all the assumptions of Corollary \ref{module cor} are satisfied.  

Now let us interpret this geometrically.  The general linear form $L$ defines a general hyperplane $H_L$ and contains no point of $Z$.  The linear forms $0 \neq \ell_i \in S$ ($1 \leq i \leq r-2$) together define a line $\lambda_\ell$ in $H_L \subset \mathbb P^r$.  Similarly we have a general hyperplane $H_{L'}$ and a line $\lambda_{\ell '} \subset H_{L'} \subset \mathbb P^r$.  The meaning of Corollary \ref{module cor} is that $\lambda_\ell$ and $\lambda_{\ell'}$ meet in a point.  Hence they span a plane $\Lambda$.  

Recall the ``Linear Lemma" of \cite{CC}:

\begin{lem}
Any set of $m$-planes such that any two of them meet in an $(m-1)$-plane, either is contained in some fixed $\mathbb P^{m+1}$ or has an $(m-1)$-plane for base locus.
\end{lem}

In our case, taking $m=1$, we have that the lines $\lambda_\ell$ as $L$ ranges over $R_1$ either lie in a 2-plane or have a point as base locus.  But in the latter case, choosing a hyperplane that avoids this point leads to a contradiction.  Thus all the lines obtained in this way lie in the plane~$\Lambda$.  Notice that $\Lambda$ does not depend on the original choice of $L$, so we may assume that $L$ is general even with respect to $\Lambda$.

Let $H$ be a general element of $[I_\Lambda]_1$, and by abuse of notation we denote by $H$ also the hyperplane in $\mathbb P^r$ defined by this linear form.  Let 
\[
\begin{array}{rcl}
\displaystyle I_{Z_1} & = & (I_Z + I_\Lambda)^{sat} \\
\displaystyle I_{Z_1|\Lambda} & = & \left ( \frac{I_Z + I_\Lambda}{I_\Lambda} \right )^{sat} \subset R/I_\Lambda \\
\displaystyle I_{Z_1|H} & = & \left ( \frac{I_Z + (H)}{(H)} \right )^{sat}  \subset R/ \langle H \rangle \\
\displaystyle I_{Z_2} & = & I_Z : H = I_Z : I_\Lambda.
\end{array}
\]
So $Z_1$ is the subset of $Z$ lying on $\Lambda$, and $Z_2$ is the subset lying off $\Lambda$. Let $L$ be a general linear form. Let $J_2 = \frac{I_{Z_2} + \langle L \rangle}{\langle L \rangle}$.  Let $h$ be the restriction of $H$ to $R/\langle L \rangle$.  We have the commutative diagram
\[
\begin{array}{ccccccccccccccc}
&&&& 0 && 0 \\
&&&& \downarrow && \downarrow \\
0 & \rightarrow & \left [ \frac{I_Z : H}{I_Z} \right ]_{n-1} & \rightarrow & [R/I_Z]_{n-1} & \stackrel{\times H}{\longrightarrow} & [R/I_Z]_n & \rightarrow & [R/(I_Z,H)]_n & \rightarrow & 0 \\
&&   &&\phantom{\scriptstyle \times L} \downarrow {\scriptstyle \times L} &&\phantom{\scriptstyle \times L} \downarrow {\scriptstyle \times L} && \\
0 & \rightarrow & \left [ \frac{I_Z : H}{I_Z} \right ]_{n} & \rightarrow & [R/I_Z]_{n} & \stackrel{\times H}{\longrightarrow} & [R/I_Z]_{n+1} & \rightarrow & [R/(I_Z,H)]_{n+1} & \rightarrow & 0 \\
&&  &&  \downarrow && \downarrow &&   \\
0 & \rightarrow & \left [ \frac{J:h}{J} \right ]_n & \rightarrow & [S/J]_n & \stackrel{\times 0}{\longrightarrow} & [S/J]_{n+1} & \rightarrow & [R/(I_Z,H,h)]_{n+1} \\
&&&& \downarrow && \downarrow \\
&&&& 0 && 0
\end{array}
\]
where the indicated zero map is multiplication by $h$.  We will also shortly consider this diagram in degree one less.  We obtain (again) $[J:h]_n = [S]_n$ and $[R/(I_Z,H,h)]_{n+1} = [S/J]_{n+1}$. 

Now, the image of the middle map $\times H$ in this diagram is $[R/I_{Z_2}]_n$, and {the} commutativity of this diagram means that reduction of $[R/I_{Z_2}]_n$ modulo $L$ is zero.  As indicated, everything continues to work in degree one less.  Thus we obtain
\begin{equation} \label{Z2}
\Delta h_{R/I_{Z_2}}(n-1) = \Delta h_{R/I_{Z_2}}(n) = 0.
\end{equation}
Now consider the diagram
\[
\begin{array}{cccccccccccc}
&& 0 && 0  \\
&& \downarrow && \downarrow \\
0 & \rightarrow & [R/I_{Z_2}]_{n-1} & \stackrel{\times H}{\longrightarrow} & [R/I_Z]_{n} & \rightarrow & [R/(I_Z,H)]_{n} & \rightarrow & 0 \\
&&  \phantom{\scriptstyle \times L} \downarrow {\scriptstyle \times L} && \phantom{\scriptstyle \times L} \downarrow {\scriptstyle \times L} && \phantom{\scriptstyle \times L} \downarrow {\scriptstyle \times L} \\
0 & \rightarrow & [R/I_{Z_2}]_{n} & \stackrel{\times H}{\longrightarrow} & [R/I_Z]_{n+1} & \rightarrow & [R/(I_Z,H)]_{n+1} & \rightarrow & 0 \\
&& \downarrow && \downarrow && \downarrow   \\
&&[S/J_2]_n && [S/J]_{n+1} && [R/(I_Z,H,L]_{n+1}  \\
&& \downarrow && \downarrow && \downarrow \\
&& 0 && 0 && 0
\end{array}
\]
We obtain:
\begin{equation}\label{induced}
\hbox{\it the induced map $[S/J_2]_n \rightarrow [S/J]_{n+1}$ is the zero map.}
\end{equation}  

Applying the Snake Lemma and looking at the rightmost column, we thus get
\[
0 \rightarrow [S/J_2]_n \rightarrow [R/(I_Z,H)]_n \stackrel{\times L}{\longrightarrow} [R/(I_Z,H)]_{n+1} \rightarrow [S/J]_{n+1} \rightarrow 0
\]
from which we conclude
\begin{equation} \label{formula}
\Delta h_{R/I_Z}(n+1) = \Delta h_{R/(I_Z,H)} (n+1) + \Delta h_{R/I_{Z_2}}(n).
\end{equation}
The same holds in degree one less thanks to our observations above.  Combining with (\ref{Z2}) we have
\begin{equation} \label{Z1}
\Delta h_{R/I_Z}(n) = \Delta h_{R/(I_Z,H)} (n) \hbox{ \ \ and \ \ } \Delta h_{R/I_Z}(n+1) = \Delta h_{R/(I_Z,H)} (n+1).
\end{equation}
Now consider the exact sequences
\[
0 \rightarrow I_{Z_2}(-1) \stackrel{\times H}{\longrightarrow} I_Z \rightarrow \frac{I_Z + \langle H \rangle }{\langle H \rangle} \rightarrow 0  \hbox{ \ \ \ and \ \ \ }  0 \rightarrow \mathcal I_{Z_2}(-1) \stackrel{\times H}{\longrightarrow} \mathcal I_Z \rightarrow \mathcal I_{Z_1 | H} \rightarrow 0.
\]
Equation (\ref{Z2}) implies $h^1(\mathcal I_{Z_2}(n-2)) = 0$.  Thus $\frac{I_Z + \langle H \rangle}{\langle H \rangle}$ is saturated in degrees $\geq n-1$.  Since the Hilbert function of $Z_1$ is the same whether viewed as a subscheme of $\mathbb P^r$ or as a subscheme of $H$ or as a subscheme of $\Lambda$, we obtain
\begin{equation} \label{main eqn}
\Delta h_{R/I_Z}(t) = \Delta h_{R/I_{Z_1}} (t) \hbox{ \ \ \ for all } t \geq n.
\end{equation}
This concludes the proof of Theorem \ref{no base locus}.

\begin{cor}
In the setting of Theorem \ref{no base locus}, at least
\[
\binom{k+1}{2} + (k+1)(n-k+2) -3
\]
points of $Z$ must lie on a plane.
\end{cor}

\begin{proof}
The approach is almost identical to the proof of Corollary \ref{count pts sect 4}, but now there is a slight twist.  Since $[J]_n$ is basepoint free, the $h$-vector cannot be
\[
(1,2,\dots,k-1,k,k,\dots,k,k-1)
\]
since in this case the component in degree $n$ is still a curve.  Thus the smallest possible $h$-vector is
\[
(1,2,\dots,k-1,k,k+1, k+1,\dots,k+1,k,k-1)
\]
and this gives the desired bound.

\end{proof}


\section{Extending Gotzmann's theorem} \label{gotzmann section}

We have noted that in \cite{BGM}, use was made of Gotzmann's Persistence Theorem, which described the behavior of the Hilbert function of $Z_1$, the subset of $Z$ lying on the base locus (under the assumption of maximal growth), assuming that no additional generators were present in its ideal.  Although it is not the focus of this paper, it is still of interest to know what the behavior of the Hilbert function can be when we have almost maximal growth and no additional generators.  To illustrate that something can be said, we consider  the situation of Theorem \ref{no base locus}.    

Note that even if $Z \subset \mathbb P^r$, its artinian reduction agrees in degrees $\geq n$ with the artinian reduction of $Z_1$, the subset of $Z$ lying on the plane $\Lambda$.  Thus without loss of generality we can assume that our almost maximal growth arises in the setting of an algebra $S/J$ where $S = k[x,y]$.

\begin{prop}
Let $S = k[x,y]$,  and let $J$ be a homogeneous ideal in $S$ such that for some integer $n$ the following hold:
\begin{itemize}
\item[(a)] The linear system defined by $[J]_n$ has no base locus;
\item[(b)] $h_{S/J}(n+1) = h_{S/J}(n) - 1$;
\item[(c)] $J$ has no minimal generators in degree $> n$.

\end{itemize}

\noindent Then for all $j \geq n$ we have $h_{S/J}(j+1) = \max \{ h_{S/J}(j) -1, 0 \}$.
\end{prop}

\begin{proof}
The minimal free resolution of $J$ has the form
\[
0 \rightarrow \mathbb F_2 \rightarrow \mathbb F_1 \rightarrow J \rightarrow 0
\]
where all summands $S(-j)$ of $\mathbb F_1$ satisfy $j \leq n$.  Condition (a) guarantees that $J$ contains a regular sequence of two forms of degree $n$.  

Now suppose that the assertion is not true.  Then there exists some $k \geq n+1$ for which 
\[
h_{S/J}(k) = h_{S/J}(k-1) - 1 \hbox{ and } h_{S/J}(k+1) \leq h_{S/J}(k) -2.
\]
Let $\mathfrak c$ be a complete intersection of type $(n,n)$ in $I$.  It links $J$ to a homogeneous artinian ideal $J' \subset S$.  We have 
\[
h_{S/J'}(2n-k-3) > h_{S/J'}(2n-k-2) = h_{S/J'}(2n-k-1) .
\]
Since the latter two represent maximal growth (or possibly 0), a result of \cite{CI} (see also \cite{GHMS} Theorem 3.4) gives that $S/J'$ has a non-zero socle element in degree $2n-k-3$.  Hence the minimal free resolution of $S/J'$ has a free summand $S(-2n+k+1)$ in the last free module.  By linkage, this means that $J$ has a minimal generator in degree $k+1 \geq n+2$, contradicting~(c).
\end{proof}

\section{Further questions}

Finally, we present some open problems which we leave for future study.  

\begin{enumerate}

\item When the base locus is of dimension $\geq 2$, can we obtain a good bound on the number of points of $Z$ lying on the base locus?

\item Are there other results extending Gotzmann's theorem besides the one given in section \ref{gotzmann section}?

\item What happens if we do not assume that $Z$ is reduced?  Can we still obtain similar results?  In this case, $Z_1$ becomes the subscheme of $Z$ lying on the base locus, which may or may not be reduced.

\item Can we obtain similar results for $Z$ of higher dimension?

\item  What happens when $Z$ is a reduced set of points in uniform position?  In \cite{BGM} several nice consequences were obtained.  Perhaps similar results can be found here.  In particular, must the base locus be irreducible and must all the points of $Z$ lie on it?

\item We believe that Theorem \ref{n+1 choose n} and Theorem \ref{other AMG} can be extended to ``maximal growth minus 2,'' and indeed to ``maximal growth minus $\ell$" for $\ell \leq r-2$.  Specifically, there should be $\ell+1$ possibilities for the dimension of the base locus in general.

\item In \cite{CCDG} and \cite{CCi}, when the set $Z\subset \pp^3$ is the general 
hyperplane section of an irreducible curve in $\pp^4$, a result similar to Theorem 
\ref{no base locus} is obtained, with a different approach. More generally, when there is a 
large monodromy group acting on $Z$, the existence of a surface of degree 
$\ell$ containing $Z$ is proven when ``maximal growth minus $\ell$" holds, 
for some $\ell$. We wonder if the approach introduced above can extend the results
of \cite{CCDG} and \cite{CCi} to any sets of points in (very?) uniform position.

\item As we mentioned in the Introduction, we have in mind some applications of our results on almost maximal growth to the study of symmetric tensors. What other applications can we obtain from these results?  

\end{enumerate}

\end{document}